\documentclass[a4paper, 10pt]{article}
\usepackage{amsmath, amsthm, amssymb}
\usepackage[T1]{fontenc}
\usepackage[utf8]{inputenc}
\usepackage{lmodern}
\usepackage{bbm}
\usepackage[normalem]{ulem}

\usepackage[colorlinks=true,
linkcolor=blue,
urlcolor=blue,
citecolor=red]{hyperref}

\usepackage[T1]{fontenc}

\usepackage{mwe} 
\usepackage{caption}
\usepackage{amsfonts,mathrsfs}
\usepackage{amsmath, amsthm, amssymb}

\usepackage{float}
\usepackage{subfigure}
\newfloat{figure}{H}{lof}
\newfloat{table}{H}{lot}
\floatname{figure}{\figurename}
\floatname{table}{\tablename}

\usepackage{mathtools}

\hypersetup{linkcolor=blue}

\usepackage{enumerate}
\usepackage[svgnames]{xcolor}
\usepackage{tcolorbox}
\usepackage{colortbl}
\usepackage{nicematrix}

\usepackage{float}
\newfloat{figure}{H}{lof}
\newfloat{table}{H}{lot}
\floatname{figure}{\figurename}
\floatname{table}{\tablename}

\numberwithin{equation}{section}
\newtheorem{theorem}{Theorem}[section]

\newtheorem{assumption}[theorem]{Assumption}

\newtheorem{definition}[theorem]{Definition}

\newtheorem{lemma}[theorem]{Lemma}

\newtheorem{proposition}[theorem]{Proposition}
\newtheorem{remark}[theorem]{Remark}

\numberwithin{equation}{section}

\renewcommand{\d}{\mathrm{d}}

\newcommand{\R}{\mathbb{R}}

\newcommand{\N}{\mathbb{N}}

\definecolor{QG}{named}{blue}

\begin{document}
	\title{\textbf{Population dynamics model for aging}}

	\author{\textsc{Jacques Demongeot$^{(1)}$ and  Pierre Magal$^{(2),}$\footnote{ \href{mailto:pierre.magal@u-bordeaux.fr}{pierre.magal@u-bordeaux.fr}} }\\
		{\small \textit{$^{(1)}$Université Grenoble Alpes, AGEIS EA7407, }} \\
		{\small \textit{	F-38700 La Tronche, France.}} \\
		{\small \textit{$^{(2)}$Univ. Bordeaux, IMB, UMR 5251, F-33400 Talence, France.}} \\
		{\small \textit{CNRS, IMB, UMR 5251, F-33400 Talence, France.}}
	}
	\maketitle

\begin{abstract}
	The chronological age used in demography describes the linear evolution of the life of a living being. The chronological age cannot give precise information about the exact developmental stage or aging processes an organism has reached. On the contrary, the biological age (or epigenetic age) represents the true evolution of the tissues and organs of the living being. Biological age is not always linear and sometimes proceeds by discontinuous jumps. These jumps can be positive (we then speak of rejuvenation) or negative (in the event of premature aging), and they can be dependent on endogenous events such as pregnancy (negative jump) or stroke (positive jump) or exogenous ones such as surgical treatment (negative jump) or infectious disease (positive jump). The article proposes a mathematical model of the biological age by defining a valid model for the two types of jumps (positive and negative). The existence and uniqueness of the solution are solved, and its temporal dynamic is analyzed using a moments equation. We also provide some individual-based stochastic simulations. 
\end{abstract}

\medskip 


\section{Introduction}
Over the past decade, scientists have studied several indicators of the health status of individuals. Chronological age, the time since birth, can be considered one of them. Several other indicators are required to improve the understanding of the health status of individuals. These indicators will combine CpG, DNA methylation (DNAm) and age (chronological age), health, and lifestyle outcomes. In the present article, we will call this indicator \textit{biological age}, or \textit{epigenetic age} \cite{Jain}\cite{Bernabeu}.   

\medskip 
Aging is generally considered irreversible due to the severe damage to the cell resulting in a gradual loss of functions and increasing fragility until final death. Recent publications proved that this process can be reversible. For example, biological age increases due to stress or traumatism and decreases in the recuperation phase during post-partum or after stress-induced or surgery-induced cell aging \cite{Poganik} (for mice).

\medskip 
This possibility of forward and backward changes in biological age due to specific events during an individual's life (pregnancy, stress, surgery, traumatism, etc.) has been considered in various approaches. For example, in \cite{Demongeot}, the calculation of biological age is based on estimating the time left to live depending on the number of undifferentiated cells remaining in the stem cell reservoir of the organs providing the patient's vital functions.

\medskip 
The present article aims to present a model for a population of patients. Our model mainly consists of continuous growth of the biological age together with some models for jumps (forward and backward) in the biological age to derive a mean behavior at a population's level. The continuous growth of the biological age corresponds to the biological age when no accidents or jumps occur. Therefore, in this case, the biological age grows at the same speed as the chronological age. In our model, the difference between biological age and chronological age consists of modeling random jumps forward or backward when the patient becomes sick or recovers from a disease such as cancer. In our model, the disease's severity corresponds to the jumps' amplitude. Here it is assumed to increase linearly with biological age.   In conclusion, we will propose a larger model class for the jumps. The general framework of models is the so-called Levy process \cite{Ycart}\cite{Applebaum}. In  Huang et al.\cite{Huang}, a mathematical model based on a stochastic differential equation is proposed to model the dynamic of biological age at the level of a single patient. 

\medskip 
The plan of the paper is the following. Section \ref{Section2} is devoted to the biological background of aging. In section  \ref{Section3}, we present an aging model with a rejuvenation mechanism only. In section \ref{Section4}, we present an aging model with a premature aging mechanism only. In section  \ref{Section5}, we combine rejuvenation and premature aging mechanisms. Section \ref{Section6} is devoted to numerical simulation, and section \ref{Section7} is to the discussion. 

\section{Biological background}
\label{Section2}

Due to microscopic cellular events, the health status of patients the life is first able to regenerate the tissue destroyed. When the chronological age increases, the patients are less capable of repairing such a tissue, leading to phenotypic medical symptoms. The gravity of the medical symptoms is also more and more severe when patients' chronological age increases. The medical symptoms result from an accumulation threshold at the microscopic level of cell delations. This justifies the model with jumps (aging and rejuvenating jump) in biological age with an amplitude that increases in the life course.

\medskip 
We can consider that, even if in fine a loss of function is due to the dysfunctioning (or even death) of organs or even to the death of the whole organism, it results from a continuous disappearance of differentiated cells not replaced by undifferentiated cells (which are finite in number at birth and whose proliferation is bounded by Hayflick's limit), the event signaling accelerated aging or rejuvenation is a discrete phenomenon linked to a phenotypic symptom appearing at a precise moment corresponding to a biological age jump in the proposed mathematical models.

\medskip 
Due to microscopic cellular events, the health status of patients the life is first able to regenerate the tissue destroyed. When the chronological age increases, the patients are less capable of repairing such tissue, leading to phenotypic medical symptoms. The gravity of these medical symptoms is also more and more severe when patients' chronological age increases. The medical symptoms result from an accumulation at the microscopic level of cell deletions.

\medskip 
The determination of longevity is a problem that affects all species. The determining factors are multiple and are of two kinds, endogenous and exogenous. Among the endogenous determinants, we find, for example, i) the microbiome, whose importance has been demonstrated for species that have longevity comparable to the human species, such as crocodiles with a maximal life span of about 100 years \cite{Siddiqui} and ii) the genome, whose influence is decisive in family diseases such as progeria, a rare genetic disorder \cite{Talukder}. Progeria accelerates aging, reduces the maximal human life span, and presents early symptoms like osteoporosis and hair loss. The exogenous factors are due to the environment and come from the food, which often contains pesticides and oxidizing components, the stress (especially occupational stress and emotional stress), infectious diseases, chronic diseases (namely diabetes, neurodegenerative and cardiovascular diseases), cancer, accidental musculoskeletal trauma, and surgical operations, etc.)  \cite{Poganik}. For example, an acute viral disease can cause a loss of 3 109 cells in 1 day, equivalent to 2 weeks of additional aging, because normal aging causes only a loss of 2 108 cells per day \cite{Nguyen}. Another example is given by the pre-dementia, in which the observed loss of cells can be caused by multiple etiologies, among which the loss of limbic and hypothalamic cell connections contributing to altering the sense of satiety (causing the adipocytes depletion), alterations in cardiomyocytes insulin sensitivity, or age-related regulatory changes in carbohydrate metabolism of the liver cells \cite{Knopman}. The loss of cells due to apoptosis (not compensated by the mitoses remaining before the Hayflick's limit \cite{Demongeot}\cite{Hayflick}) in the different organs (brain, heart, liver, etc.) involved in these dysfunctions causes not the marginal organ death, but the whole organism death due to an involution (in cell size and number) of these organs incompatible with the survival of the whole organism.

\medskip 
As the accelerating aging factors, the rejuvenation factors also have two sources: the endogenous origin comes from cellular repair processes, in particular DNA repair mechanisms \cite{Kim}\cite{Stead} which prevent the abnormal apoptosis due to DNA damages (caused by radiations, abortive mitoses, etc.). The heterogeneous origin is due to recovery processes during the healing time after a disease or an exhausting physiological event like pregnancy \cite{Poganik}\cite{Murray-Sherrat}\cite{Kashdan}.

\medskip

The chronological age is the commonly used age which is the time since birth. To describe a more realistic lifetime expectancy, we now focus on the biological age, which is much more difficult to define since it is multi-factorial. 

\medskip 
According to Gerdhem et al. \cite{Gerdhem},  biological age is a commonly used term without clear definitions but is routinely used to describe patients. Biological age may differ from chronological age and correlates to gait, muscle strength, and balance. The general appearance is considered when estimating the biological age.

\medskip 
In Demongeot \cite{Demongeot}, the biological age is defined as the real age of interface tissues (like intestinal endothelium, alveoli epithelium, and skin epithelium) and differs from the chronological age classically used in demography, but unable to give useful information about the exact stage in development or aging an organ or more globally an organism has reached. Biological age is essentially determined by the number of divisions remaining before Hayflick's limit (the maximal number of mitoses for a cell lineage inside a dedicated differentiated cellular population) of the tissues of a critical organ (critical in the sense that its loss causes the death of the whole organism). We refer to Shay and Wright \cite{Shay-Wright} for more information. The critical organs have vital functions in interaction with the environment (protection, homeothermy, nutrition, and respiration) and present a rapid turnover (the total cell renewal time is in mice equal to 3 weeks for the skin, 1.5 days for the intestine, 4 months for the alveolate and 11 days for mitochondrial inner membrane) conditioning their biological age. 
\begin{figure}[H]
	\centering
	\includegraphics[scale=0.24]{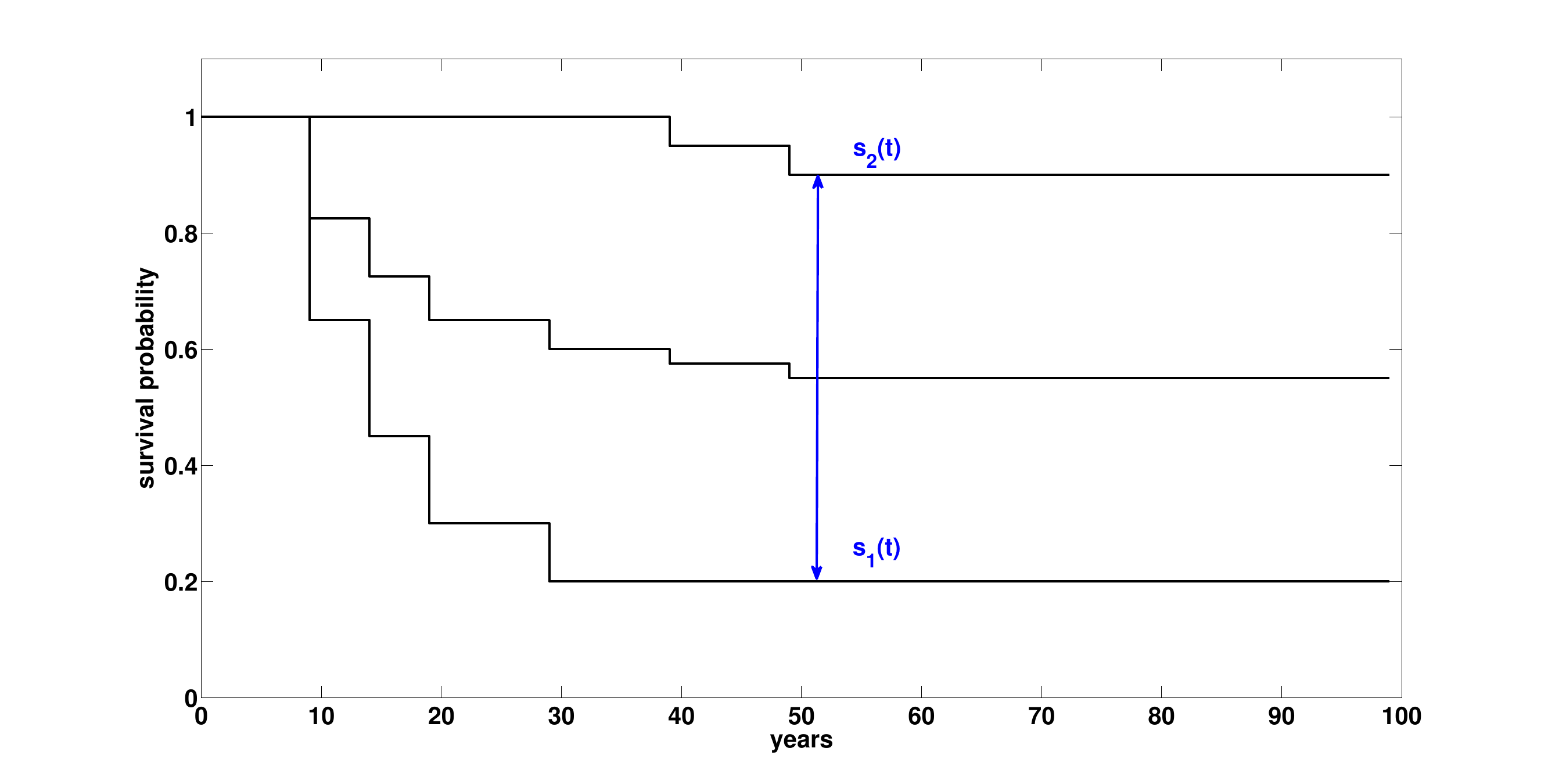}
	\caption{\textit{Kaplan-Meier survival confidence ($95\%$) curves. The region in between the curve $s_1(t)$ and the curve $s_2(t)$ corresponds to $95\%$ of individuals having their probability of survival to time $t$ between $s_1(t)$ and $s_2(t)$ ($95\%$ confidence limits).}}
	\label{Fig1}
\end{figure}

Another estimation of the biological age is the epigenetic age based on the expression level of certain genes like
ELMSAN1 (also known as MIDEAS) and their transcription factors (e.g., mir4505 for MIDEAS) \cite{Bernabeu}. Figure \ref{Fig2} gives the statistical 
relationship between the epigenetic age and the chronological age showing that the probability of observing backward or forward 
jumps in epigenetic age and the intensity of these jumps increase linearly with chronological age. We refer to \cite{JPH} and \cite{SL} for more results on this subject. 

\begin{figure}[H]
	\centering
	\includegraphics[scale=0.43]{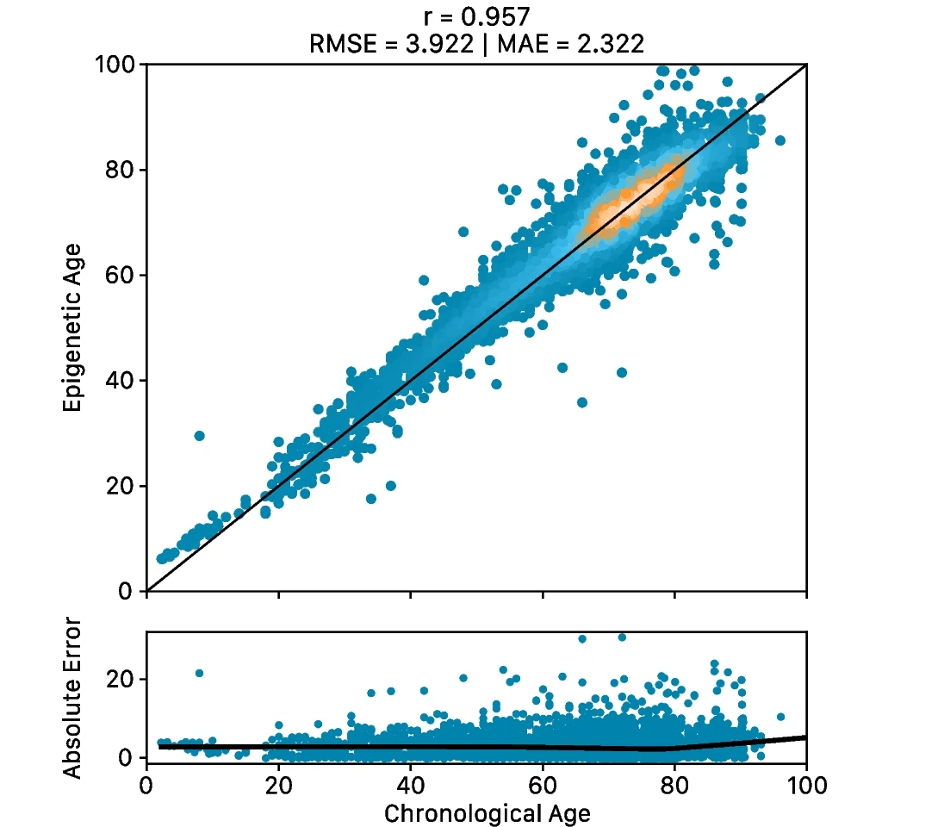}
	\caption{\textit{The statistical relationship between chronological and epigenetic ages (top) with an indication of absolute error (bottom) in the prediction of epigenetic age by the chronological one. This figure is taken from \cite{Bernabeu}.}}
	\label{Fig2}
\end{figure}

\section{Aging model with rejuvenation mechanism only}
\label{Section3}
To describe the rejuvenation of a population of individuals structured with respect to biological age $b \in [0,+ \infty)$, we consider $b \mapsto n(t,b)$ the population density at time $t$. That is 
$$
\int_{b_1}^{b_2} u(t,b) \d b
$$
is the number of individuals at time $t$ with biological age between $b_1$ and $b_2$. 

\medskip 
To model the rejuvenation, we will use the following system of partial differential equations 
\begin{equation}  \label{3.1}
	\left\lbrace
	\begin{array}{rl}
		\partial_t u(t,b)+\partial_b u(t,b) &=  - \tau_+ u(t,b)+ \tau_+ \left( 1+\delta_+ \right)  u\left(t, \left( 1+\delta_+ \right) b\right), \vspace{0.2cm}\\
		u(t,0)&=0, \vspace{0.2cm}\\
		u(0,b)&=u_0(b) \in L^1_+\left((0,\infty),\R \right),
	\end{array}
	\right.
\end{equation}
where $\tau_+>0$ is the rate of rejuvenation (i.e. $\dfrac{1}{\tau_+}$ is for a single individual,  the average time between two rejuvenations jumps), $\delta_+ \geq 0$ is the fraction of biological age $b$ (after rejuvenation) that should be added to $b$ to obtain $\widehat{b}= \left( 1+\delta_+ \right) b$ the biological age before rejuvenation. 

\medskip 
In this above model, the term $\partial_b u(t,b) $ corresponds to a drift term with a constant velocity (which represents the classical chronological aging in the absence of perturbations jumps), the term $ \tau_+  u(t,b) $ corresponds to the flow of individuals which rejuvenate  (i.e. having a jump in biological age) at time $t$. This flow is given by $ \tau_+ \int_{0}^{\infty} u(t,b) \d b $. That is, 
$$
\int_{t_1}^{t_2}  \tau_+ \int_{0}^{\infty} u(t,b) \d b \d t 
$$ 
is the number of individuals which rejuvenate    in between $t_1$ and $t_2$.

\medskip 
More precisely, when a rejuvenation occurs an individual having a biological age $b$ after a rejuvenation's jump,  its biological age was 
$$
\widehat{b}= \left( 1+\delta_+ \right) b>b +\delta_+ b>b,
$$
before the rejuvenation's jump.

\medskip 
In other words, a rejuvenating individual with a biological age $b$ after a rejuvenation's jump, was  an older individual with biological age   $\widehat{b}= \left( 1+\delta_+ \right) b$ before rejuvenation's jump. That is also equivalent to say that, rejuvenating individual starting from the biological age $\widehat{b}$  end-ups with a biological age $b= \dfrac{1}{1+\delta_+ } \widehat{b}$ after rejuvenation.  So this individual looses a fraction 
$$
1-\dfrac{1}{1+\delta_+ }=\dfrac{\delta_+ }{1+\delta_+ }
$$ 
of its biological age $\widehat{b}$ before rejuvenation's jump.  

\medskip 
Setting
\begin{equation} \label{3.2}
	g_+=1+\delta_+, 
\end{equation}
the system \eqref{3.1} becomes 
\begin{equation}  \label{3.3}
	\left\lbrace
	\begin{array}{rl}
		\partial_t u(t,b)+\partial_b u(t,b) &=  - \tau_+ u(t,b) + \tau_+ g_+ u\left(t, g_+ b\right),  \vspace{0.2cm}\\
		u(t,0)& =0 , \vspace{0.2cm}\\
		u(0,b)&=u_0(b) \in L^1_+\left((0,\infty),\R \right).
	\end{array}
	\right.
\end{equation}

\medskip 
\noindent \textbf{Integral formulation of the solutions: }
By integrating the first equation of \eqref{3.3} along the characteristics (i.e. $t-a$ constant) we obtain the following 
\begin{equation} \label{3.4}
	u(t,b)
	=	\left\{ 
	\begin{array}{rll}
		e^{- \tau_+ t} u_0(b-t)+&v(t,b), & \text{if } t <b, \vspace{0.3cm}\\
		&v(t,b), & \text{if } t>b,
	\end{array}
	\right. 
\end{equation} 
where 
\begin{equation} \label{3.5}
	v(t,b)=	\left\{ 
	\begin{array}{lll}
		\int_{0}^{t} 	e^{- \tau_+ \left(t-\sigma\right)}  \tau_+ g_+ u\big(  \sigma, g_+ \left(b-t +\sigma\right)\big) \d \sigma, & \text{if } t <b, \vspace{0.3cm}\\
		\int_{0}^{b} 	e^{- \tau_+ \left(b-\sigma\right)}  \tau_+ g_+ u\big( t-b+ \sigma, g_+ \sigma \big) \d \sigma, & \text{if } t>b.  
	\end{array}
	\right. 
\end{equation} 
As a consequence  of the above formula \eqref{3.4} and \eqref{3.5}, and by applying fixed argument in suitable subspaces of $L^1$  with compact supports, we obtain the following lemma. 
\begin{lemma}\label{LE3.1} If 
	$$
	{\rm Support} \left(u_0\right) \subset [0,b^{\star} ],
	$$
	then 
	$$
	{\rm Support} \left(u(t,.)\right) \subset [0,b^{\star}+t], \forall t>0. 
	$$
	
\end{lemma}

\medskip 
\noindent \textbf{Abstract Cauchy problem formulation:} We refer to \cite{Webb} \cite{Thieme} \cite{Cazenave} \cite{Engel-Nagel} \cite{Magal-Ruan} for more results on semigroup theory and their application to age structured models. We consider  
$$
X= L^1 \left(\left(0, \infty\right), \R \right),
$$
endowed its standard norm 
$$
\Vert \phi \Vert_{L^1}=  \int_{0}^{\infty}\vert \phi \left(\sigma \right)\vert \d \sigma . 
$$
We consider $A:D(A) \subset X \to X$ the linear operator defined by 
$$
A\phi= -\phi' 
$$ 
with 
$$
D(A)= \left\{ \phi \in W^{1,1} \left(\left(0, \infty\right)  , \R \right): \phi \left(0\right)=0 \right\}. 
$$
We consider $B: X \to X$ the bounded linear operator defined by 
$$
B_g \phi(x)=g \,  \phi \left( g \, x \right), \text{ for } x \geq 0.,
$$
where $g >0$. 
Then $B_g$ is an isometric bounded linear operator. That is  
$$
\Vert B_g\phi \Vert_{L^1}=\Vert \phi \Vert_{L^1}, \forall \phi \in L^1 \left(\left(0, \infty\right), \R \right). 
$$
The problem \eqref{3.3} can be reformulated as an abstract Cauchy problem 
\begin{equation} \label{3.6}
	\left\{ 
	\begin{array}{l}
		u'(t)=Au(t)-\tau_+u(t)+ \tau_+ B_{g_+}u(t), \text{ for } t \geq 0,\vspace{0.2cm} \\
		\text{with}\vspace{0.2cm} \\
		u(0)=u_0 \in L^1_+ \left(\left(0, \infty\right), \R \right).	
	\end{array}
	\right.
\end{equation}

\begin{lemma}
	The linear operator $A$ is the infinitesimal general of $\left\{T_A \left(t\right)\right\}_{t \geq 0} $ the strongly continuous semigroup of linear operators, defined by 
	\begin{equation} \label{3.7}
		T_A\left(t\right)\left( \phi\right) \left(a\right) =\left\{ 
		\begin{array}{ll}
			\phi\left(a-t\right), &\text{ if } a>t,  \vspace{0.2cm}\\
			0,& \text{ if } a<t. 
		\end{array}
		\right. 
	\end{equation}
\end{lemma}

\begin{definition}
	We will say that $u \in C\left(\left[ 0, \infty\right),  L^1_+ \left(\left(0, \infty\right), \R \right)\right)$ is a \textbf{mild solution} of  \eqref{3.6} if 
	$$
	\int_0^t u(\sigma) \d \sigma \in D(A), \forall t \geq 0, 
	$$
	and 
	$$
	u(t)= u_0 + A \int_{0}^{t}u(\sigma) \d \sigma+  \int_{0}^{t}-\tau_+u(\sigma)+ \tau_+ B_{g_+}u(\sigma) \d \sigma, \forall t \geq 0.
	$$
\end{definition}
\begin{theorem} \label{TH3.4} For each $u_0  \in L^1_+ \left(\left(0, \infty\right), \R \right),$ the Cauchy problem \eqref{3.6} admits a unique mild solution which is the unique continuous function $u \in C\left(\left[ 0, \infty\right),  L^1_+ \left(\left(0, \infty\right), \R \right)\right)$ satisfying the fixed point problem 
	
	\begin{equation} \label{3.8}
		u(t)= T_{A-\tau_+I}(t)u_0+ \int_{0}^{t} T_{A-\tau_+I}(t-\sigma)\tau_+ B_{g_+}u(\sigma) \d \sigma, \forall t \geq 0,
	\end{equation}
	where 
	\begin{equation} \label{3.9}
		T_{A-\tau_+I}(t)=e^{-\tau_+ t}T_{A}(t) ,  \forall t \geq 0. 
	\end{equation}
\end{theorem}
\begin{remark} \label{REM3.5}
	The  variation of constant formula  \eqref{3.8} is an abstract reformulation of the  formulas \eqref{3.4} and   \eqref{3.5}.  
\end{remark}
\medskip 
\noindent \textbf{Moments' formulation of the PDE:} 
Define 
$$
E_k(u(t))=\int_{0}^{\infty} \sigma^k u(t, \sigma) \d \sigma, \forall k \geq 0, 
$$
with 
$$
E_0(u(t))=\int_{0}^{\infty}  u(t, \sigma) \d \sigma,
$$
assuming that the integrals are well defined. 
\begin{theorem} \label{TH3.6}The rejuvenation model \eqref{3.1} has a unique non-negative mild solution.  We have  
	\begin{equation} \label{3.10}
		\frac{d}{dt}	E_0(u(t)) =0, 
	\end{equation}
	and the model preserves the total mass (number) of individuals. That is 
	\begin{equation}\label{3.11}
		\int_{0}^{\infty} u(t, \sigma) \d \sigma=  	\int_{0}^{\infty} u_0(\sigma) \d \sigma, \forall t \geq 0, 
	\end{equation}
	Moreover, the higher moment satisfies the following system of ordinary differential equations 
	\begin{equation}\label{3.12}
		\frac{d}{dt}	E_k(u(t)) =k \, E_{k-1}(u(t))- \chi_k \, E_{k}(u(t)), \forall k \geq 1, 
	\end{equation}
	where 
	\begin{equation*}
		\chi_k= \tau_+  \left(1-\dfrac{1}{\left(1+\delta_+\right)^k}\right)>0, \forall k \geq 1, 
	\end{equation*}
	and 
	\begin{equation*}
		\lim_{k \to +\infty} 	\chi_k= \tau_+>0. 
	\end{equation*}
	If we denote the equilibrium solution of \eqref{3.12}  by 
	\begin{equation*}
		\overline{E}_k=\left(  \prod_{j=1}^{k}  \dfrac{j}{\chi_j}  \right)\int_{0}^{\infty} u_0(\sigma) \d \sigma,  \forall k \geq 1,
	\end{equation*}
	consequently  from \eqref{3.12},  we have the following convergence result 
	\begin{equation*}
		\lim_{t \to \infty} E_k(u(t)) =\overline{E}_k,  \forall k \geq 1.
	\end{equation*}
\end{theorem}
\begin{proof}
	\begin{equation*}
		\begin{array}{rl}
			\displaystyle		\int_{0}^{\infty} \sigma^k \partial_t  u(t, \sigma) \d \sigma+	\int_{0}^{\infty} \sigma^k   \partial_b u(t, \sigma) \d \sigma &=  - \tau_+ 	\int_{0}^{\infty}  \sigma^k u(t, \sigma) \d \sigma \\
			&+ \tau_+  	\int_{0}^{\infty} \sigma^k u\left(t, g_+  \sigma \right)   g_+   \d \sigma,
		\end{array}		
	\end{equation*}
	by integrating by parts the second integral, and by making a change of variable in the last integral, we obtain 
	\begin{equation*}
		\frac{d}{dt}	E_k(u(t)) =k \, E_{k-1}(u(t))-  \tau_+  \left(1-\dfrac{1}{\left(1+\delta_+\right)^k}\right)	\int_{0}^{\infty}  E_k(u(t)). 
	\end{equation*}
\end{proof}
\section{Aging model with premature aging mechanism only}
\label{Section4}
To model the premature aging, we will use the following system of partial differential equations 
\begin{equation}  \label{4.1}
	\left\lbrace
	\begin{array}{rl}
		\partial_t u(t,b)+\partial_b u(t,b) &=  - \tau_- u(t,b)+ \tau_- \left( 1-\delta_- \right)  u\left(t, \left( 1-\delta_- \right) b\right),  \vspace{0.2cm}\\
		u(t,0)&=0, \vspace{0.2cm}\\
		u(0,b)&=u_0(b) \in L^1_+\left((0,\infty),\R \right),
	\end{array}
	\right.
\end{equation}
where $\tau_->0$ is the rate of premature aging (i.e. $\dfrac{1}{\tau_-}$ is for a single individual, the average time between two premature aging jumps), $\delta_- \in ( 0, 1)$ is the fraction of biological age $b$ (after premature aging) that should be subtracted to $b$ to obtain $\widehat{b}= \left( 1-\delta_- \right) b$ the biological age before premature aging. 

\medskip 
In this above model, the flow of individuals with premature aging  (i.e. having a jump in biological age) at time $t$ is given by $ \tau_- \int_{0}^{\infty} u(t,b) \d b $. That is, 
$$
\int_{t_1}^{t_2}  \tau_- \int_{0}^{\infty} u(t,b) \d b \d t 
$$ 
is the number of individual with premature aging  in between $t_1$ and $t_2$. 

\medskip 
More precisely, when a rejuvenation occurs an individual having a biological age $b$ after a premature aging's jump, its biological age was 
$$
\widehat{b}= \left( 1-\delta_- \right) b>b -\delta_- b<b. 
$$
before the premature aging's jump. 

\medskip 
In other words, a premature aging individual with a biological age $b$ after a premature aging's jump, was  an younger individual with biological age   $\widehat{b}= \left( 1-\delta_- \right) b$ before premature aging's jump. That is also equivalent to say that, premature aging individual starting from the biological age $\widehat{b}$  end-ups with a biological age $b= \dfrac{1}{1-\delta_- } \widehat{b}$ after premature aging.  So this individual gains a fraction 
$$
\dfrac{1}{1-\delta_- }-1=\dfrac{\delta_- }{1-\delta_- }
$$ 
of its biological age $\widehat{b}$ before premature aging's jump. 

\medskip 
Setting
\begin{equation} \label{4.2}
	g_-=1-\delta_-, 
\end{equation}
the system \eqref{4.1} becomes 
\begin{equation}  \label{4.3}
	\left\lbrace
	\begin{array}{rl}
		\partial_t u(t,b)+\partial_b u(t,b) &=  - \tau_- u(t,b)+ \tau_- g_- u\left(t, g_- b\right), \vspace{0.2cm}\\
		u(t,0)&=0, \vspace{0.2cm}\\
		u(0,b)&=u_0(b) \in L^1_+\left((0,\infty),\R \right).
	\end{array}
	\right.
\end{equation}
\section{Aging model}
\label{Section5}
The full model with both rejuvenation and premature aging reads as follows 
\begin{equation}  \label{5.1}
	\left\lbrace
	\begin{array}{rl}
		\partial_t u(t,b)+\partial_b u(t,b) &= \ - \left(\tau_-+\tau_+ \right) u(t,b) \vspace{0.2cm}\\
		&\quad + \tau_- \, g_-  \,  u\left(t, g_- \, b\right) \vspace{0.2cm}\\
		&\quad +\tau_+ \,  g_+ \,  u\left(t, g_+ \, b\right), \vspace{0.2cm}\\
		u(t,0)&=0, \vspace{0.2cm}\\
		u(0,b)&=u_0(b) \in L^1_+\left((0,\infty),\R \right).
	\end{array}
	\right.
\end{equation}
and we make the following assumption on the parameters of the system.
\begin{assumption} \label{ASS5.1} We assume that 
	$$
	\tau_+>0,  \, \tau_->0, \text{ and  } g_+ >1 >g_- >0.
	$$ 
\end{assumption}

\medskip 
\noindent \textbf{Volterra integral formulation:} By integrating the first equation of \eqref{5.1} along the characteristics (i.e. $t-a$ constant) we obtain the following 
\begin{equation*} \label{5.3}
	u(t,b)
	=	\left\{ 
	\begin{array}{lll}
		e^{- \left(\tau_-+\tau_+ \right) t} u_0(b-t)+&v_+(t,b)+v_-(t,b), & \text{if } t <b, \vspace{0.3cm}\\
		&v_+(t,b)+v_-(t,b), & \text{if } t>b,
	\end{array}
	\right. 
\end{equation*} 
where 
\begin{equation*} \label{5.4}
	{\small
		v_+(t,b)=	\left\{ 
		\begin{array}{lll}
			\int_{0}^{t} 	e^{- \left(\tau_-+\tau_+ \right)  \left(t-\sigma\right)}  \tau_+ g_+ u\big(  \sigma, g_+ \left(b-t +\sigma\right)\big) \d \sigma, & \text{if } t <b, \vspace{0.3cm}\\
			\int_{0}^{b} 	e^{- \left(\tau_-+\tau_+ \right)  \left(b-\sigma\right)}  \tau_+ g_+ u\big( t-b+ \sigma, g_+ \sigma \big) \d \sigma, & \text{if } t>b, 
		\end{array}
		\right. 
	}
\end{equation*} 
and
\begin{equation*} \label{5.5}
	{\small
		v_-(t,b)=	\left\{ 
		\begin{array}{lll}
			\int_{0}^{t} 	e^{- \left(\tau_-+\tau_+ \right)  \left(t-\sigma\right)}  \tau_- g_- u\big(  \sigma, g_- \left(b-t +\sigma\right)\big) \d \sigma, & \text{if } t <b, \vspace{0.3cm}\\
			\int_{0}^{b} 	e^{- \left(\tau_-+\tau_+ \right)  \left(b-\sigma\right)}  \tau_- g_- u\big( t-b+ \sigma, g_- \sigma \big) \d \sigma, & \text{if } t>b.  
		\end{array}
		\right. 
	}
\end{equation*}

\medskip 
\noindent \textbf{Abstract Cauchy problem reformulation:} The problem \eqref{5.1} can be reformulated as an abstract Cauchy problem 
\begin{equation} \label{5.2}
	\left\{ 
	\begin{array}{l}
		u'(t)=Au(t)-\left(\tau_-+\tau_+\right) u(t)+ \left(\tau_- B_{g_-}+ \tau_+ B_{g_+}\right)u(t), \text{ for } t \geq 0,\vspace{0.2cm} \\
		\text{with}\vspace{0.2cm} \\
		u(0)=u_0 \in L^1_+ \left(\left(0, \infty\right), \R \right).	
	\end{array}
	\right.
\end{equation}


\begin{definition} \label{DE5.2}
	We will say that $u \in C\left(\left[ 0, \infty\right),  L^1_+ \left(\left(0, \infty\right), \R \right)\right)$ is a \textbf{mild solution} of  \eqref{5.2} if 
	$$
	\int_0^t u(\sigma) \d \sigma \in D(A), \forall t \geq 0, 
	$$
	and for each $t \geq 0,$
	$$
	u(t)= u_0 + A \int_{0}^{t}u(\sigma) \d \sigma+  \int_{0}^{t}-\left(\tau_-+\tau_+\right) u(\sigma)+ \ \left(\tau_- B_{g_-}+ \tau_+ B_{g_+}\right)u(\sigma) \d \sigma.
	$$
\end{definition}
\begin{theorem} \label{TH5.3} Let Assumption \ref{ASS5.1} be satisfied. For each $u_0  \in L^1_+ \left(\left(0, \infty\right), \R \right),$ the Cauchy problem \eqref{5.2} admits a unique mild solution  $u \in C\left(\left[ 0, \infty\right),  L^1_+ \left(\left(0, \infty\right), \R \right)\right)$ which is the unique continuous function satisfying the fixed point problem for each $t \geq 0$, 
	
	\begin{equation*} \label{5.7}
		\begin{array}{rl}
			u(t)= & \quad 	T_{A-\left(\tau_-+\tau_+\right) I}(t)u_0 \vspace{0.2cm} \\
			&+ \int_{0}^{t} 	T_{A-\left(\tau_-+\tau_+\right) I}(t-\sigma) \left(\tau_- B_{g_-}+ \tau_+ B_{g_+}\right)u(\sigma) \d \sigma,  	 
		\end{array}		
	\end{equation*}
	where 
	\begin{equation*} \label{5.8}
		T_{A-\left(\tau_-+\tau_+\right) I}(t)=e^{-\left(\tau_-+\tau_+\right) t}T_{A}(t) ,  \forall t \geq 0. 
	\end{equation*}
\end{theorem}

\medskip 
\noindent \textbf{Moments formulation:} We obtain the following result using similar arguments in for Theorem  \ref{TH3.6}. 
\begin{theorem} \label{TH5.4} Let Assumption \ref{ASS5.1} be satisfied.  The rejuvenation and premature aging model \eqref{5.1} has a unique non-negative mild solution.  We have  
	\begin{equation*} \label{5.9}
		\frac{d}{dt}	E_0(u(t)) =0, 
	\end{equation*}
	and the model preserves the total mass (number) of individuals. That is 
	\begin{equation} \label{5.3}
		\int_{0}^{\infty} u(t, \sigma) \d \sigma=  	\int_{0}^{\infty} u_0(\sigma) \d \sigma, \forall t \geq 0. 
	\end{equation}
	Moreover, the higher moment satisfies the following system of ordinary differential equations 
	\begin{equation}  \label{5.4}
		\frac{d}{dt}	E_k(u(t)) =k \, E_{k-1}(u(t))- \chi_k \, E_{k}(u(t)), \forall k \geq 1, 
	\end{equation}
	where
	\begin{equation} \label{5.5}
		\chi_k= \tau_+  \left(1-\dfrac{1}{\left(1+\delta_+\right)^k}\right)+\tau_-  \left(1-\dfrac{1}{\left(1-\delta_-\right)^k}\right), \forall k \geq 1,
	\end{equation}
	and 
	\begin{equation} \label{5.6}
		\lim_{k \to +\infty} \chi_k=-\infty. 	
	\end{equation}
	\begin{remark} \label{REM5.5}
		For $k=1$, and $\tau_+=\tau_-=\tau/2$, and $ \delta_+=\delta_-=\delta \in (0,1)$, then by using the formula \eqref{5.5} we obtain 
		$$
		\begin{array}{ll}
			\chi_1&= \tau_+  \left(1-\dfrac{1}{\left(1+\delta_+\right)}\right)+\tau_-  \left(1-\dfrac{1}{\left(1-\delta_-\right)}\right) \vspace{0.2cm}\\
			& =\tau/2 \left[\dfrac{\delta}{1+\delta} -  \dfrac{\delta}{1-\delta} \right] \vspace{0.2cm}\\
			&=\tau/2 \left[\dfrac{\delta (1- \delta)- \delta (1+\delta)}{1-\delta^2} \right] 
		\end{array}
		$$
		and 
		$$
		\chi_1=-\tau \dfrac{\delta^2}{1-\delta^2}.
		$$
		Next, by using formula  \eqref{5.4}, we obtain 
		\begin{equation}  \label{5.4}
			\frac{d}{dt}	E_1(u(t)) = E_{0}(u_0)+ \tau \dfrac{\delta^2}{1-\delta^2} \, E_{1}(u(t)), \forall k \geq 1. 
		\end{equation}
		The means value of the distribution is 
		$$
		M(u(t))=\dfrac{E_1(u(t))}{E_0(u(t))}
		$$
		and since $t \to E_0(u(t))$ is a constant function we obtain 
		consequently, we obtain 
		$$
		M(u(t))' = 1+\tau \dfrac{\delta^2}{1-\delta^2} \,  M(u(t)). 
		$$
		We conclude that, 
		$$
		\lim_{t \to \infty} M(u(t))=+\infty, 
		$$
		and	 the smaller $\delta$ is, the closer  the mean value $M(u(t))$ growth like the time $t$. 
	\end{remark}
	\begin{lemma}  Let Assumption \ref{ASS5.1} be satisfied.  Define 
		$$
		x_{\max}=\frac{\ln \left( \tau_+\ln q_+\right) -\ln \left(-\tau_- \ln q_{-} \right)  }{\ln g_+ -\ln g_-}.  
		$$
		Consider the function $\chi : [0, + \infty) \to \R$  defined  by 
		$$
		\chi(x)= \tau_+  \left( 1-g_+^{-x}  \right)+\tau_-  \left(1-g_-^{-x}\right).
		$$
		There we have the following alternative 
		\begin{itemize}
			\item [{\rm (i)}] If $\tau_+ \ln g_+ +	\tau_- \ln g_->0$,  then the map $\chi$ is first increases from $\chi(0)=0$ to $\chi(x_{\max} )>0$ on $[0, x_{\max} ]$, and then decreases from $\chi(x_{\max} )>0$ to $-\infty$ on $[x_{\max}, \infty )$. 
			\item [{\rm (ii)}] If $\tau_+ \ln g_+ +	\tau_- \ln g_-\leq 0$, then  the map $\chi$ decreases from $0$ to $-\infty$ on $[0, \infty )$. 
		\end{itemize}
		
	\end{lemma}
\end{theorem}
\begin{proposition} \label{PROP5.4}
	Assume that there exists $k_0 \in \N$,  such that 
	\begin{equation*}
		\chi_k>0, \forall k=1 , \ldots, k_0, 	
	\end{equation*}
	and 
	\begin{equation*}
		\chi_k<0, \forall k > k_0.	
	\end{equation*}
	If we denote the equilibrium solution of \eqref{5.4}  by 
	\begin{equation}  \label{5.14}
		\overline{E}_k=\left(  \prod_{j=1}^{k}  \dfrac{j}{\chi_j}  \right)\int_{0}^{\infty} u_0(\sigma) \d \sigma,  \forall k=1 , \ldots, k_0 ,
	\end{equation}
	consequently  from \eqref{5.4},  we have the following convergence result 
	\begin{equation}   \label{5.15}
		\lim_{t \to +\infty} E_k(u(t)) =\overline{E}_k,  \forall k=1 , \ldots, k_0, 
	\end{equation}
	and
	\begin{equation}  \label{5.16}
		\lim_{t \to +\infty} E_k(u(t)) =+\infty,  \forall k >  k_0.
	\end{equation}	
\end{proposition}

\medskip 
\noindent \textbf{Equilibrium solution:}  An equilibrium solution satisfies some delay equation with both advance and retarded delay. That is,  
\begin{equation}  \label{5.10}
	\left\lbrace
	\begin{array}{rl}
		\overline{u}'(b) &= \ - \left(\tau_-+\tau_+ \right) \overline{u}(b) \vspace{0.2cm}\\
		&\quad + \tau_- \, g_-  \,  \overline{u}\left(g_- \, b\right) \vspace{0.2cm}\\
		&\quad +\tau_+ \,  g_+ \,   \overline{u} \left(g_+ \, b\right), \vspace{0.2cm}\\
		\overline{u}(0)&=0,  
	\end{array}
	\right.
\end{equation}
and the difficulty of solving such an equation comes from the following 
$$
g_+ \, b>b >g_- \, b, \, \forall b >0.
$$ 
It follows that, even the existence of equilibrium solution is not a classical problem to investigate. 

\medskip 
\noindent \textbf{Non existence result for exponentially decreasing equilibrium solution of \eqref{5.1}:}  Assume that $b \mapsto \overline{u}(b) $ is a non-negative continuously differentiable map satisfying system \eqref{5.10}.  

\begin{assumption} \label{ASS5.7}
	Assume that the map $b \mapsto \overline{u}(b) $ is non-negative, and non null, and  continuously differentiable map, and  satisfies the system \eqref{5.10}. Assume in addition that $ \overline{u}(b) $ is exponentially decreasing. That is, there exist two constants  $M>0$, and $\gamma>0$, such that 
	$$
	\overline{u}(b) \leq M e^{-\gamma b}, \forall b \geq 0. 
	$$ 
\end{assumption}
By using  the first equation of \eqref{5.10}, we deduce that 
$$
\vert \overline{u}'(b) \vert \leq \widetilde{M} e^{-\gamma \, \min(1,g_+, g_-) \, b}, \forall b \geq 0,
$$
for some suitable $\widetilde{M} >0$.

So, under Assumption \ref{ASS5.7},  all the moments of  $\overline{u}(b) $ and $\overline{u}'(b) $ are well defined. Moreover, by using the first equation of  \eqref{5.10}, we obtain for each $k\geq 1 $ 
\begin{equation*}
	\begin{array}{rl}
		\displaystyle		\int_{0}^{\infty} \sigma^k \overline{u}'(\sigma) \d \sigma &=  - \left(\tau_-+\tau_+ \right) 	\int_{0}^{\infty}  \sigma^k \overline{u}(\sigma) \d \sigma + \tau_+  	\int_{0}^{\infty} \sigma^k \overline{u}\left(t, g_+  \sigma \right)   g_+   \d \sigma\\
		&\quad + \tau_-  	\int_{0}^{\infty} \sigma^k \overline{u}\left(t, g_-  \sigma \right)   g_-   \d \sigma,
	\end{array}		
\end{equation*}
and since $\overline{u}(0)=0 $, we obtain by integrating by parts 
\begin{equation*}
	\begin{array}{rl}
		\displaystyle	-k	\int_{0}^{\infty} \sigma^{k-1} \overline{u}(\sigma) \d \sigma &=  - \left(\tau_-+\tau_+ \right) 	\int_{0}^{\infty}  \sigma^k \overline{u}(\sigma) \d \sigma + \tau_+  	\int_{0}^{\infty} \sigma^k \overline{u}\left(t, g_+  \sigma \right)   g_+   \d \sigma\\
		&\quad + \tau_-  	\int_{0}^{\infty} \sigma^k \overline{u}\left(t, g_-  \sigma \right)   g_-   \d \sigma,
	\end{array}		
\end{equation*}
and we obtain 
\begin{equation} \label{5.11}
	k	E_{k-1}(\overline{u})= \chi_k \, E_{k}(\overline{u}), \forall k \geq 1, \forall k \geq 1, 
\end{equation}
where $\chi_k$ is define by \eqref{5.5}. 

Under Assumption \ref{ASS5.7}, we must have 
$$
E_{k-1}(\overline{u})>0, \forall k \geq 0,
$$
and by using \eqref{5.6}, we deduce that \eqref{5.11} can not be satisfied for all $k \geq 1$ large enough. Therefore we obtain the following proposition. 

\begin{proposition} \label{PROP5.8}
	The system \eqref{5.10} has no exponential decreasing solution. That is a solution of \eqref{5.10} satisfying the Assumption \ref{ASS5.7}.
\end{proposition}

\section{Numerical simulations}
\label{Section6}
\subsection{Simulation of PDE \eqref{5.1}} 
By setting 
$$
\tau_+=\tau \, p, \text{ and }  \tau_-=\tau \, (1-p),
$$
where $\tau$ is the rate at which individual either rejuvenates of prematurely ages, and $p \in [0,1]$ is the probability of  rejuvenation and $(1-p) \in [0,1]$ is the probability of premature aging. The parameter $p$ can be interpreted as the probability of being cured in the event of illness or injury. The parameter $1-p$ at the opposite is the probability of getting injured or sick. 

By using these new parameters, we obtain  a probabilistic interpretation of the model \eqref{5.1}, and the model  \eqref{5.1} becomes 
\begin{equation}  \label{7.1}
	\left\lbrace
	\begin{array}{rl}
		\partial_t u(t,b)+\partial_b u(t,b) &= - \tau  \, u(t,b) \vspace{0.2cm}\\
		&\quad + \tau \, (1-p)\, g_-  \,  u\left(t, g_- \, b\right) \vspace{0.2cm}\\
		&\quad +\tau  \, p \,  g_+ \,  u\left(t, g_+ \, b\right) \vspace{0.2cm}\\
		u(t,0)&=0,\\
		u(0,b)&=u_0(b) \in L^1_+\left((0,\infty),\R \right),
	\end{array}
	\right.
\end{equation}
and
$$
g_+ >1 >g_- >0.
$$ 

In order to run a simulation of model \eqref{7.1}, we use stochastic simulations. We consider a population composed of a finite number $N=100 \, 000$  of individuals.  We start the simulation a time $t=0$ with  all individuals with the same age $a=20$ years. The time to the next event (rejuvenation or premature aging) follows an exponential law with parameters $1/\tau$.  The principles of the simulations are as follows: When an event occurs  we choose random one individuals; and we compute a random value between $0$ and $1$. If this value is  $p$ rejuvenation  occurs, and premature aging occurs otherwise.  At each time step the biological age increases by one time step.  

In Figure \ref{Fig3A},  we present some simulation of the model \eqref{7.1} whenever $p=0.5$,  $1/\tau=1$ years,  $ g_+=1+\delta_+=1.1$ and $g_-=1-\delta_-=0.9$.  In Figure \ref{Fig3A}, we can observe that starting from a single cohort of individuals with biological age $20$ at time $t=0$, the mean value of the distribution follows the chronological age (thanks to the fact that $p=1/2$, and $\delta_+=\delta_-=0.1$,  and the remark \ref{REM5.5}). But the density of the population deviates more and more with time. One also needs to interpret the biological age by saying that the larger the biological age is, the more people are likely to die. Therefore the more the population is dispersed around the means value, the more people (with larger ages) are likely to die earlier. People with a large biological age can be understood as people suffering from a lack of treatment for their illnesses or injuries.

\begin{figure}
	\begin{center}
		\textbf{(a)} \hspace{4cm} \textbf{(b)}\\
		\includegraphics[scale=0.13]{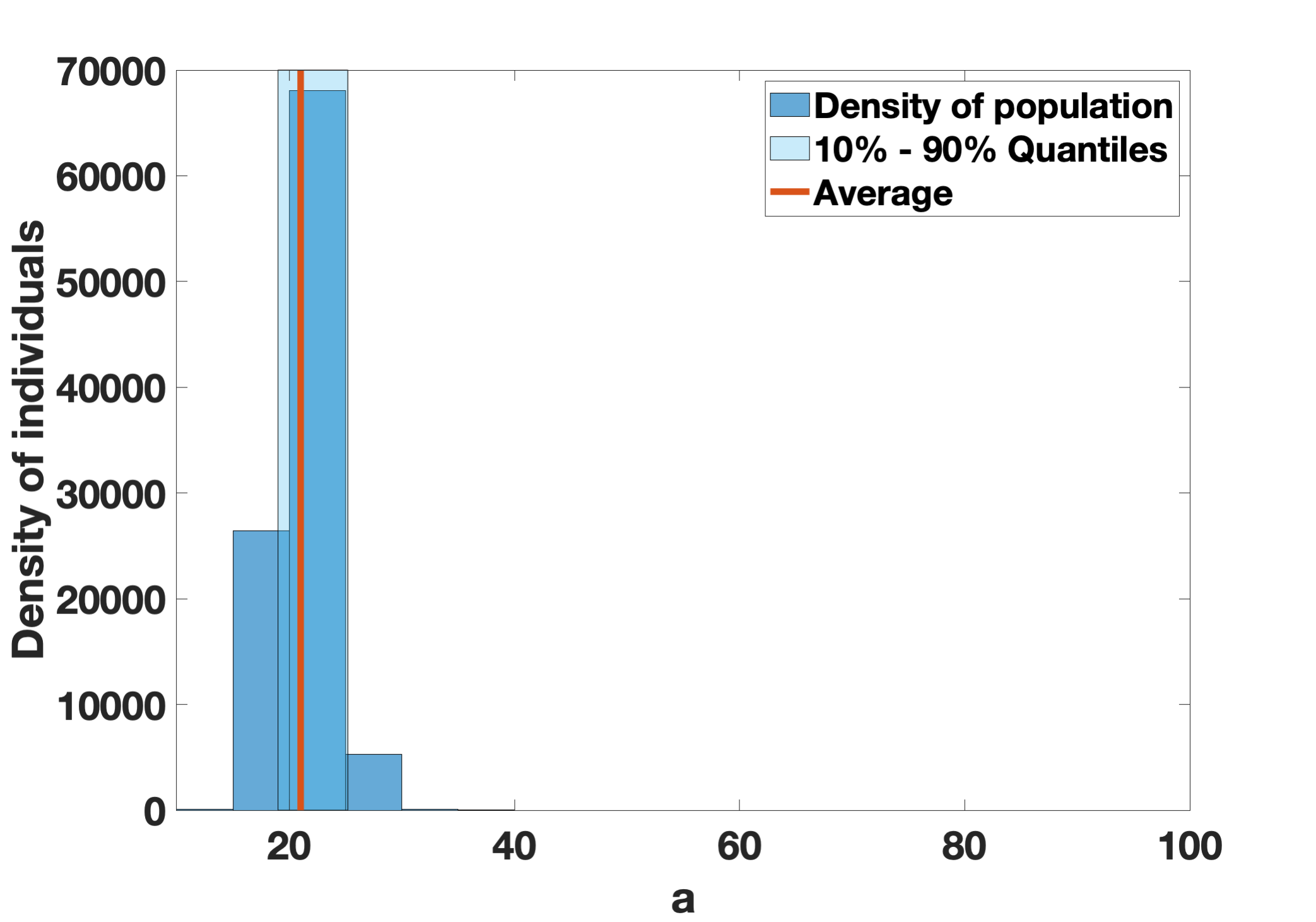}
		\includegraphics[scale=0.13]{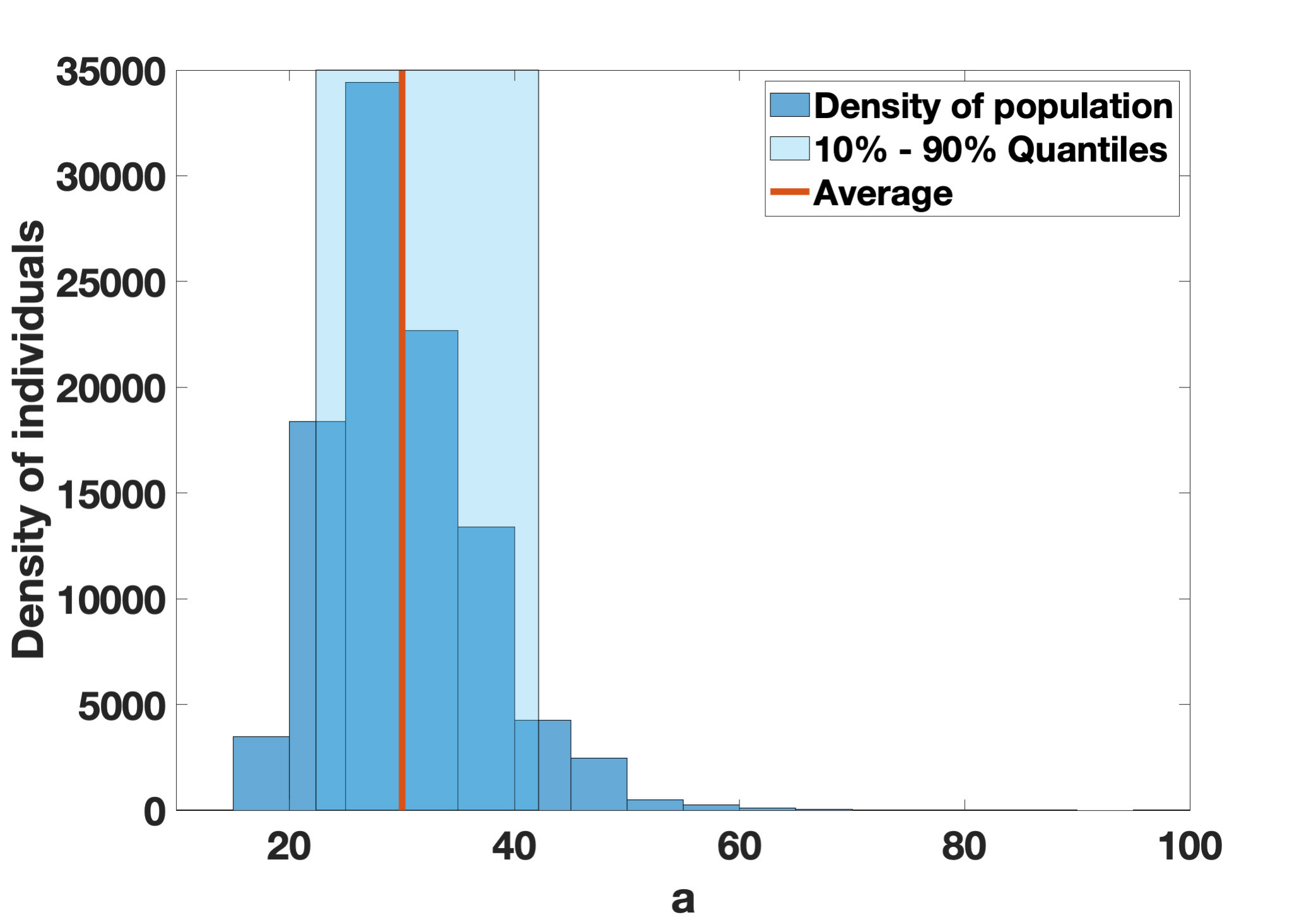}\\
		\textbf{(c)} \hspace{4cm} \textbf{(d)}\\
		\includegraphics[scale=0.13]{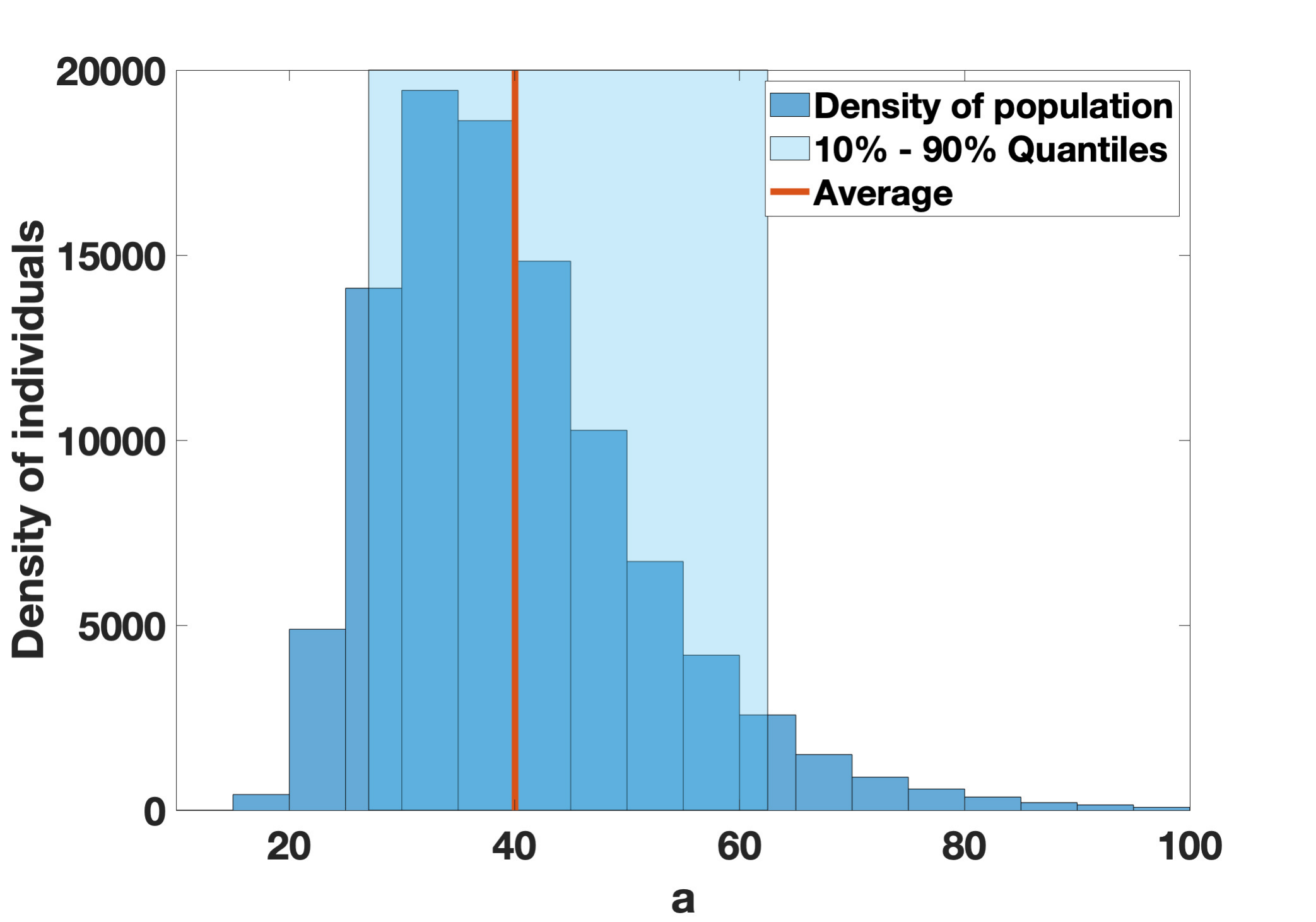}
		\includegraphics[scale=0.13]{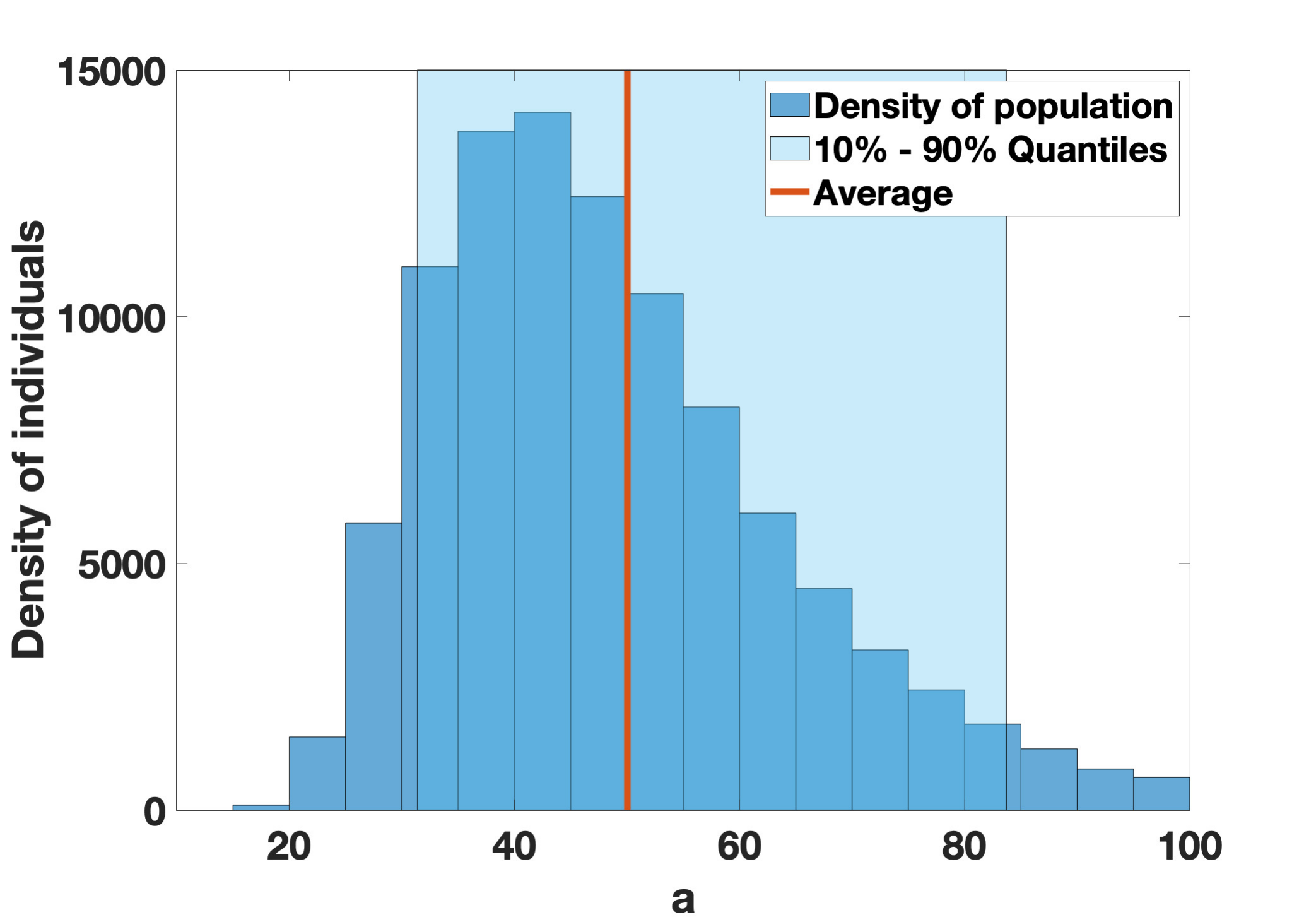}
	\end{center}
	\caption{\textit{In this figure, we plot some stochastic numerical simulations of the model \eqref{7.1} whenever $p=0.5$,  $1/\tau=1$ years,  $ g_+=1.1$ and $g_-=0.9$. We start the simulations with a cohort of $100 \, 000$ individuals all with biological age $20$ years old. The figures (a) (b) (c) (d) are respectively the distribution after $1$ year, $10$ years, $20$ years and $30$ years.    }}\label{Fig3A}
\end{figure}

In the Figures \ref{Fig3B} and \ref{Fig3C},  we investigate the influence of the parameter $p=0.25$ in Figure \ref{Fig3B},  and  $p=0.75$ in Figure \ref{Fig3C}. In Figures \ref{Fig3B} and \ref{Fig3C}, we can see that due to the dissymmetric value of $p$ the mean value no longer follows the chronological age. We can observe that the parameter $p$ plays very important in the aging process. 
\begin{figure}
	\begin{center}
		\textbf{(a)} \hspace{4cm} \textbf{(b)}\\
		\includegraphics[scale=0.13]{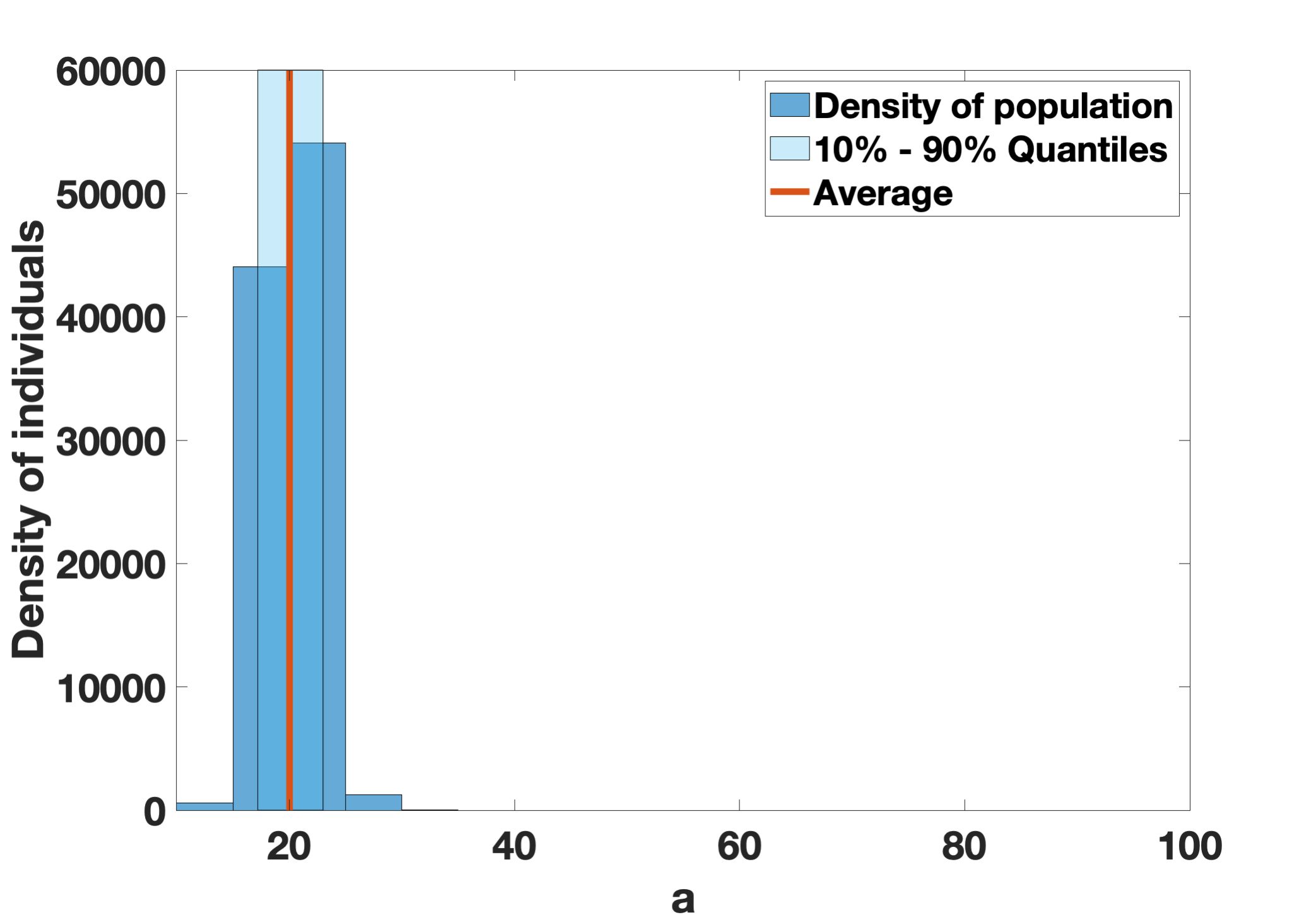}
		\includegraphics[scale=0.13]{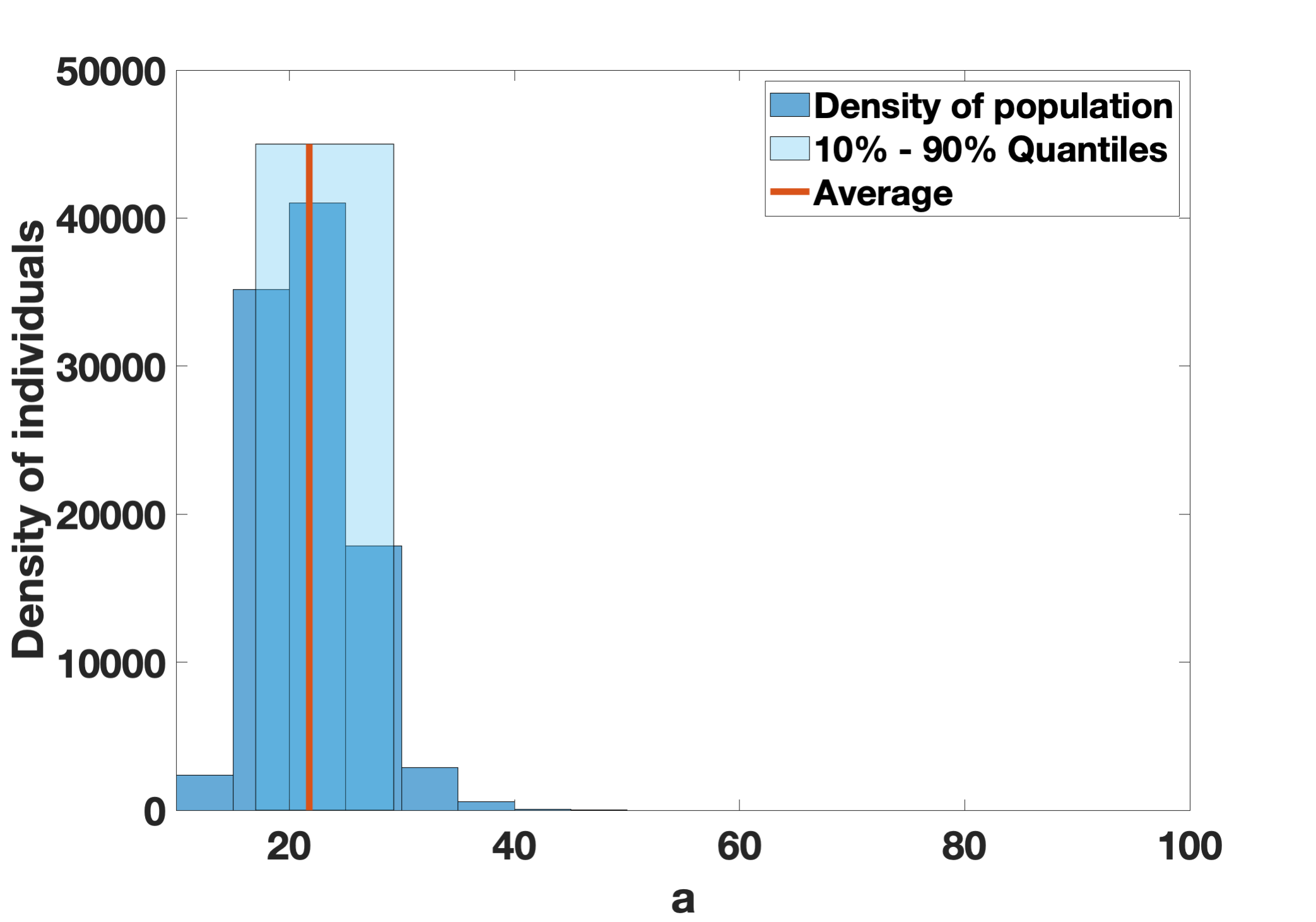}\\
		\textbf{(c)} \hspace{4cm} \textbf{(d)}\\
		\includegraphics[scale=0.13]{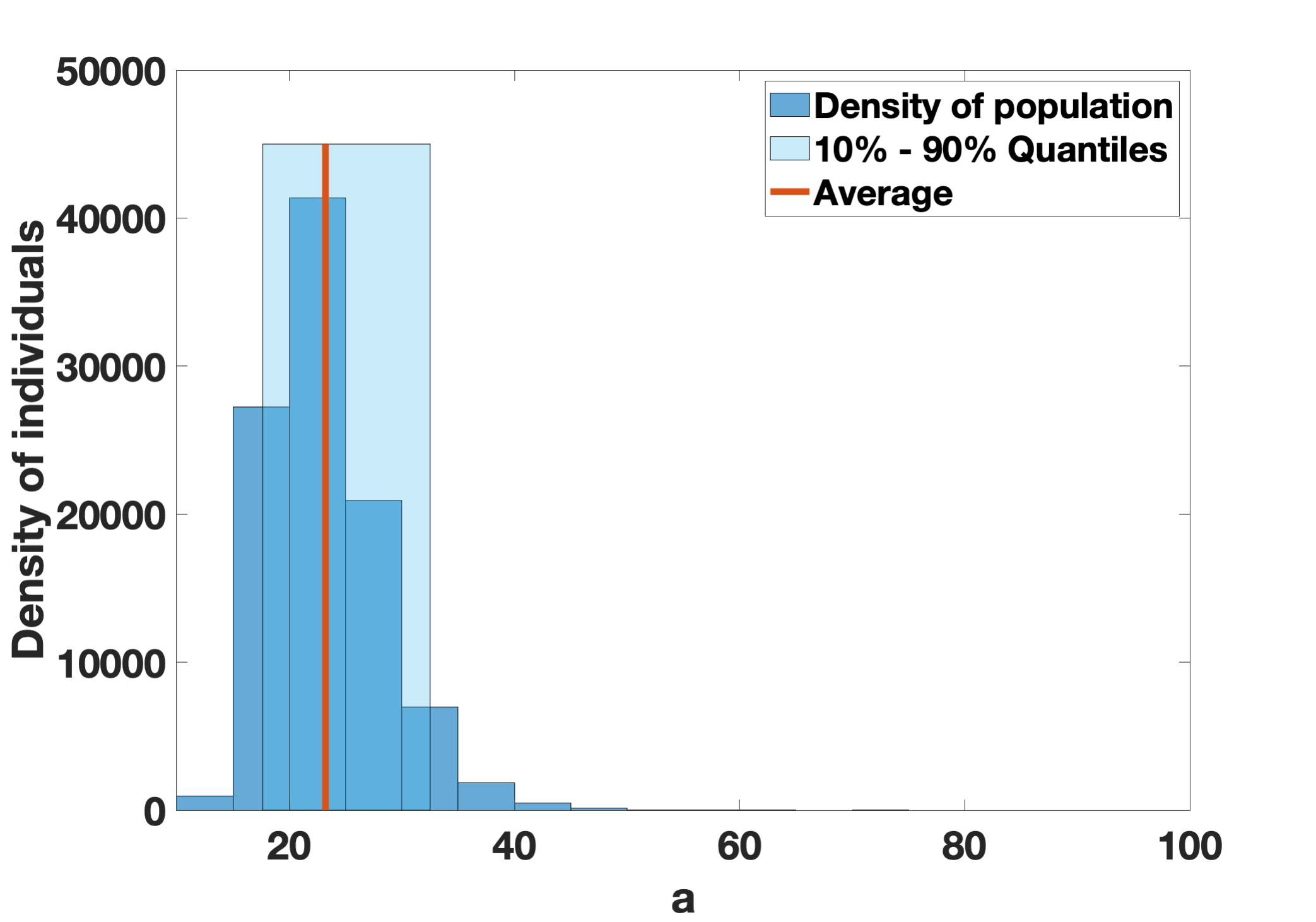}
		\includegraphics[scale=0.13]{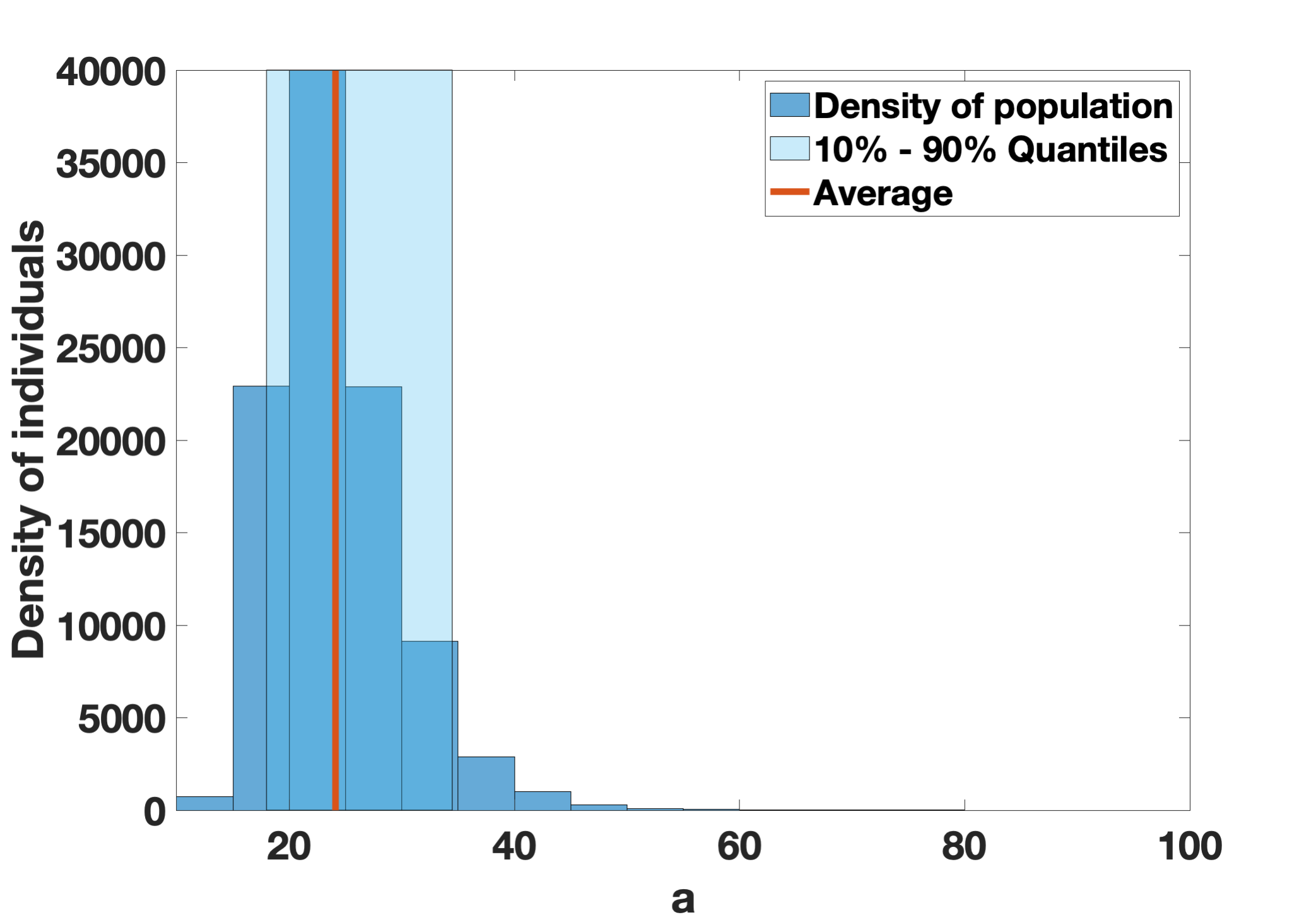}
	\end{center}
	\caption{\textit{In this figure, we plot some stochastic numerical simulations of the model \eqref{7.1} whenever $p=0.25$, the rest is the same as in Figure  \ref{Fig3A}.}}\label{Fig3B}
\end{figure}
\begin{figure}
	\begin{center}
		\textbf{(a)} \hspace{4cm} \textbf{(b)}\\
		\includegraphics[scale=0.13]{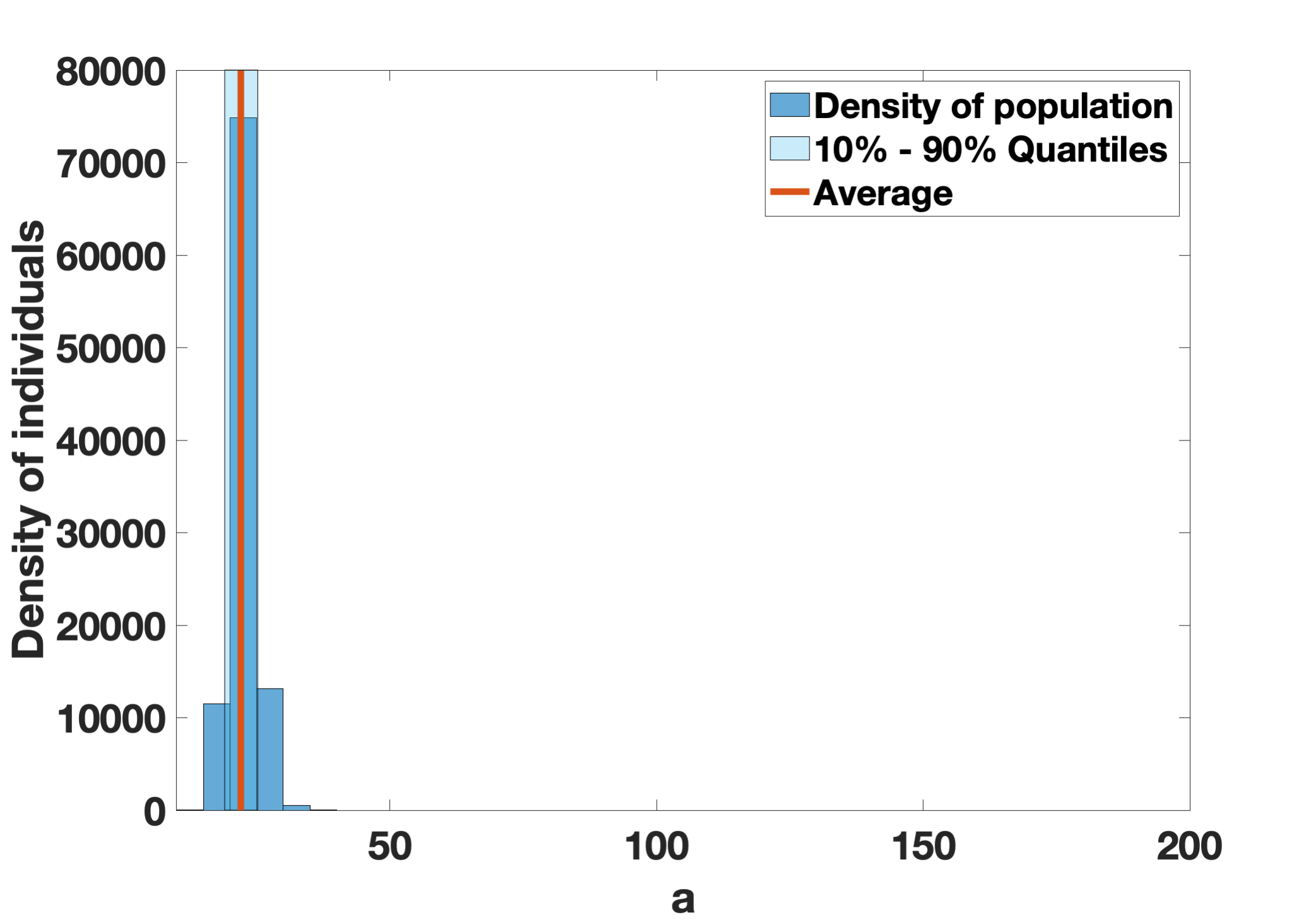}
		\includegraphics[scale=0.13]{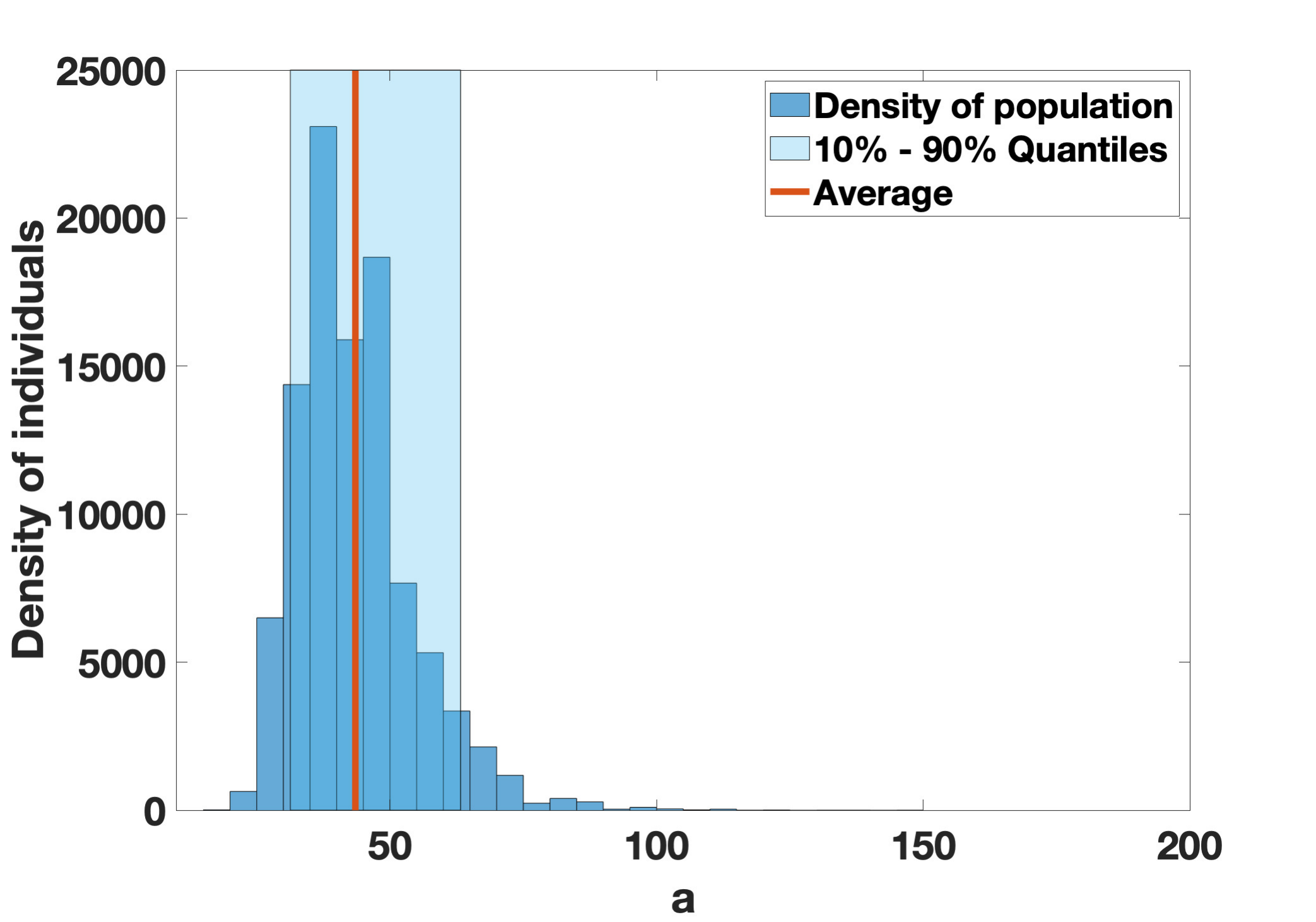}\\
		\textbf{(c)} \hspace{4cm} \textbf{(d)}\\
		\includegraphics[scale=0.13]{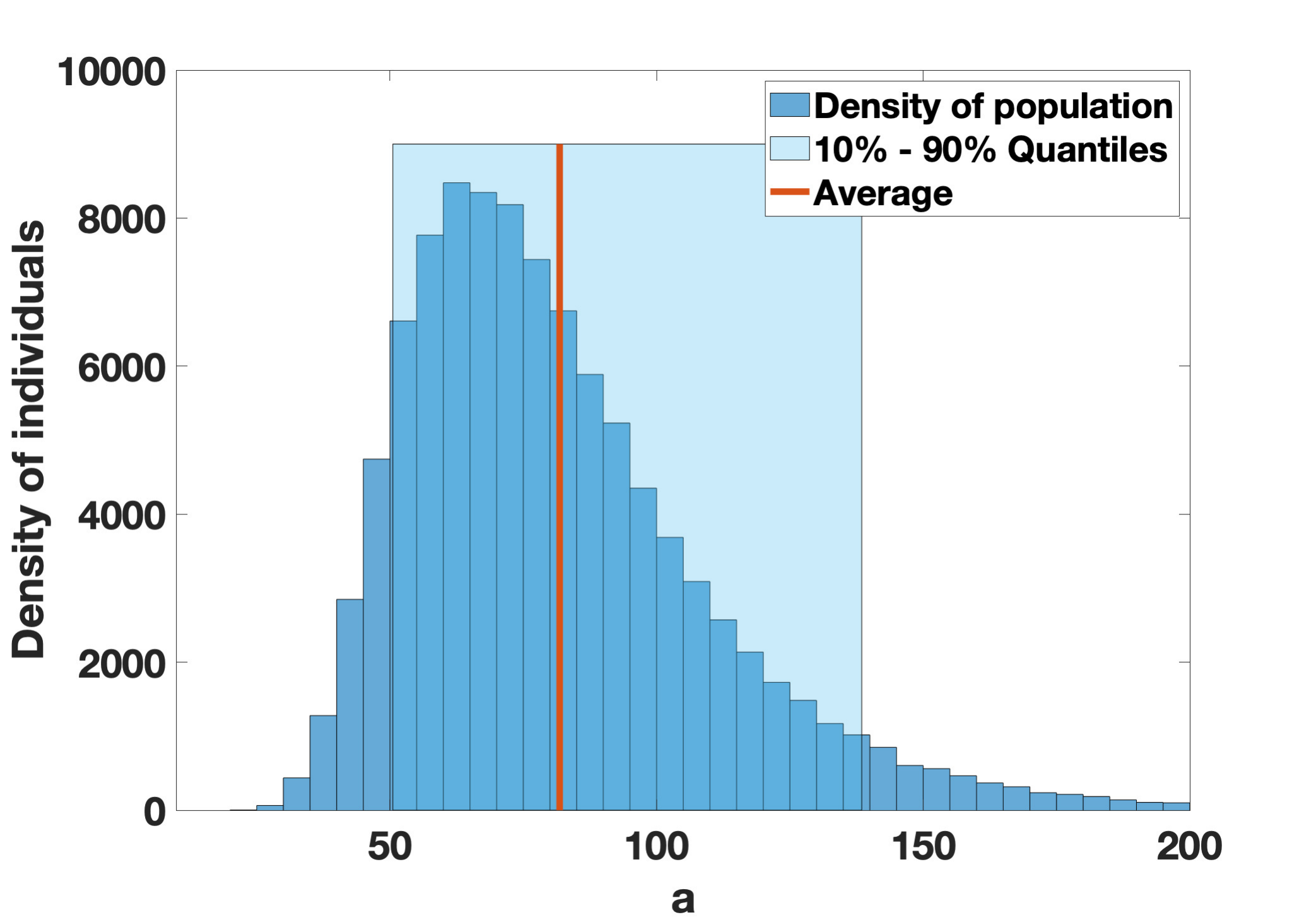}
		\includegraphics[scale=0.13]{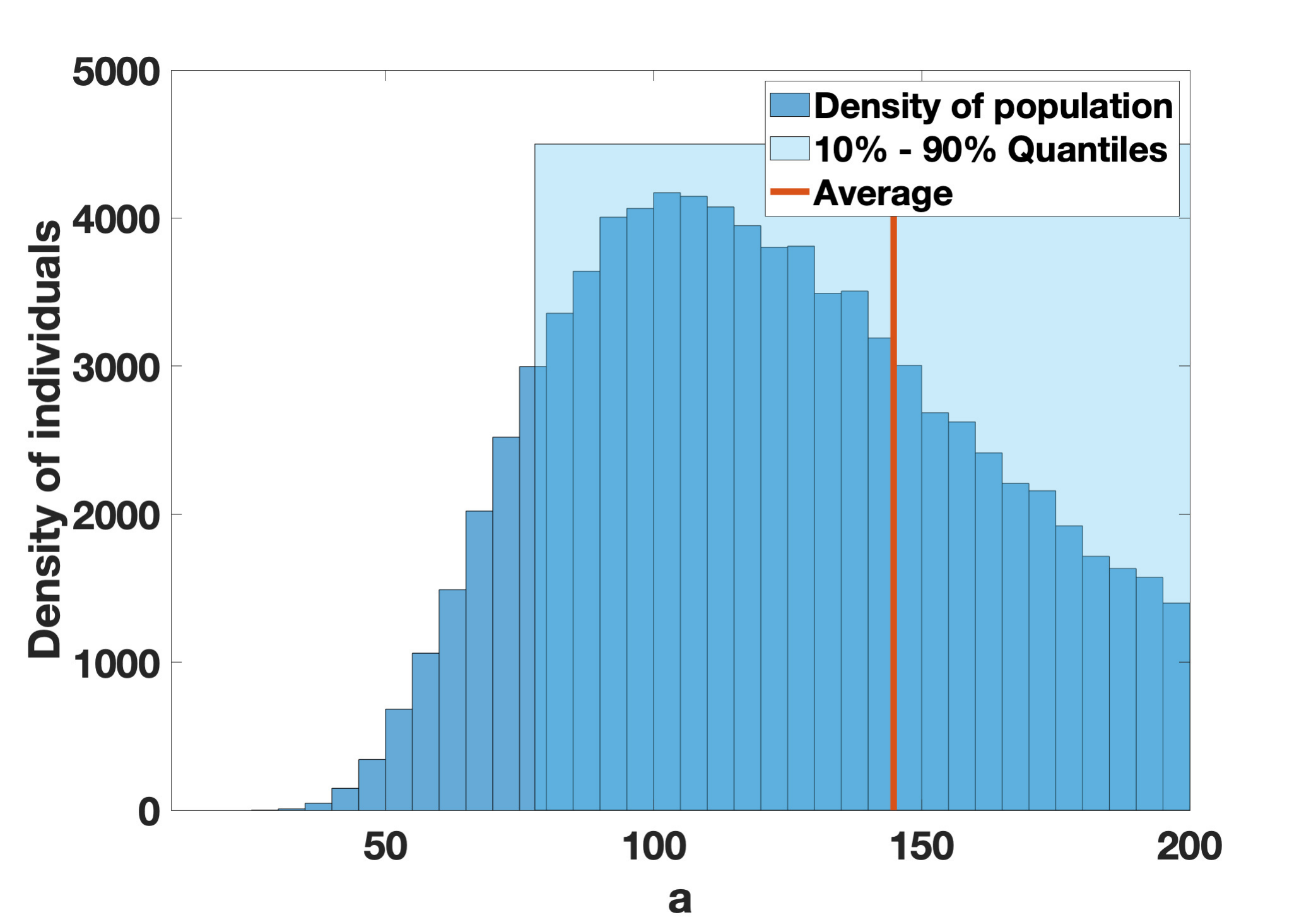}
	\end{center}
	\caption{\textit{In this figure, we plot some stochastic numerical simulations of the model \eqref{7.1} whenever $p=0.75$, the rest is the same as in Figure  \ref{Fig3A}.  }}\label{Fig3C}
\end{figure}

\subsection{Simulation of the moments equation \eqref{5.4}}
We choose 
$$
\tau_+=\tau_-=0.1, \delta_+=0.1, \delta_+=0.01, 
$$
and we obtain  
$$
g_+=1+\delta_+=1.1, \text{ and } g_-=1-\delta_-=0.99.  
$$
We use the initial distribution with compact support 
$$
u_0(b)= \dfrac{2}{5} \max \left( 0, b \left(1-b/10\right) \right).
$$
In Figure \ref{Fig3}, the initial values $k \mapsto E_k(u_0)$ increase with $k$ (see (b)). Moreover, the components $k \mapsto E_k(u(t))$ keep the same order for all $t>0$ (see (d)). Moreover by using (c) and the Theorem \ref{TH5.3}, we deduce that $k_0 \in [60,80]$, such that if $k$ is below $k_0$ the components converge to $	\overline{E}_k$, and   if $k$ is above $k_0$ the components go to $+ \infty$. 
\begin{figure}
	\begin{center}
		\textbf{(a)} \hspace{5cm} \textbf{(b)}\\
		\includegraphics[scale=0.14]{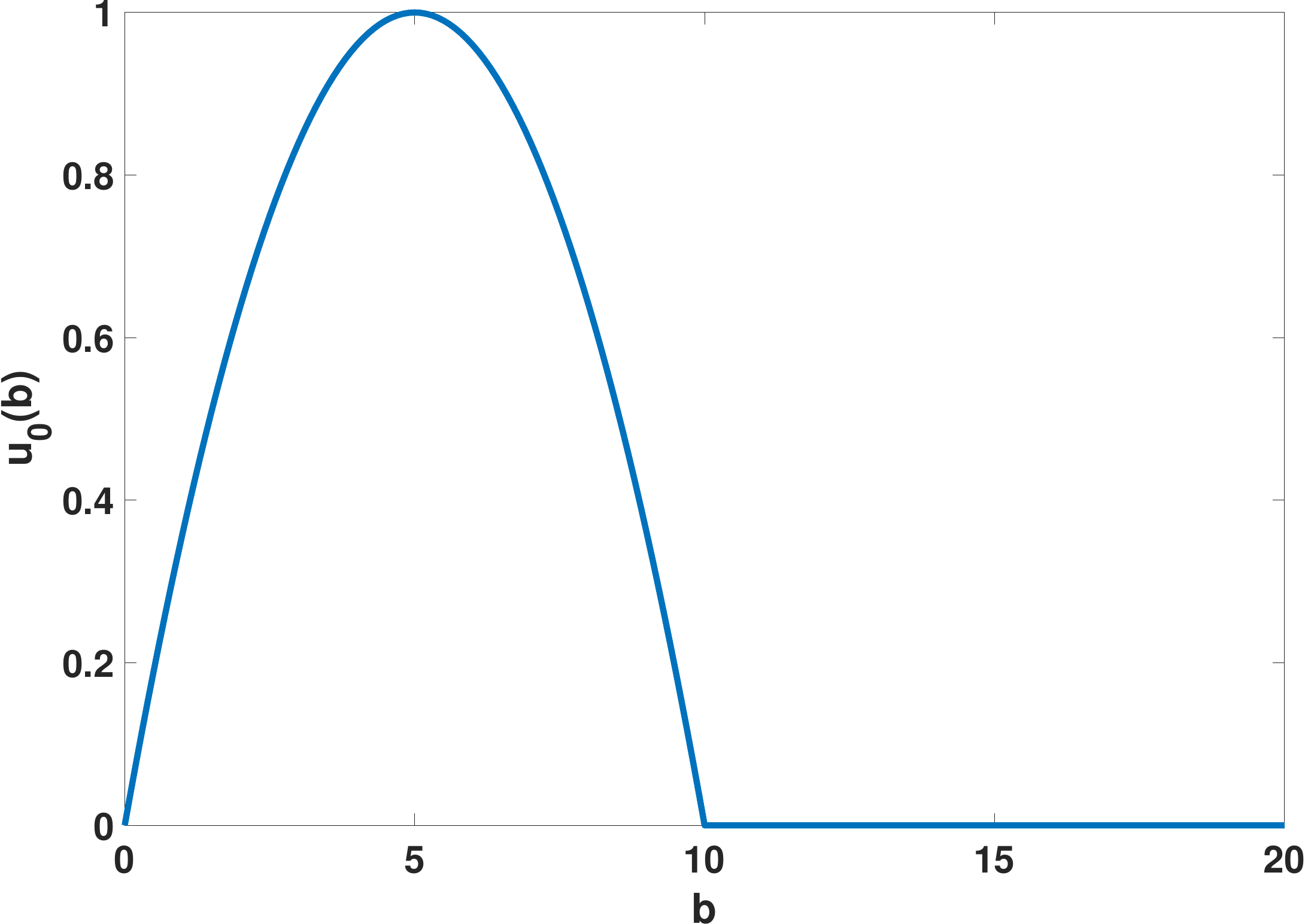}
		\includegraphics[scale=0.14]{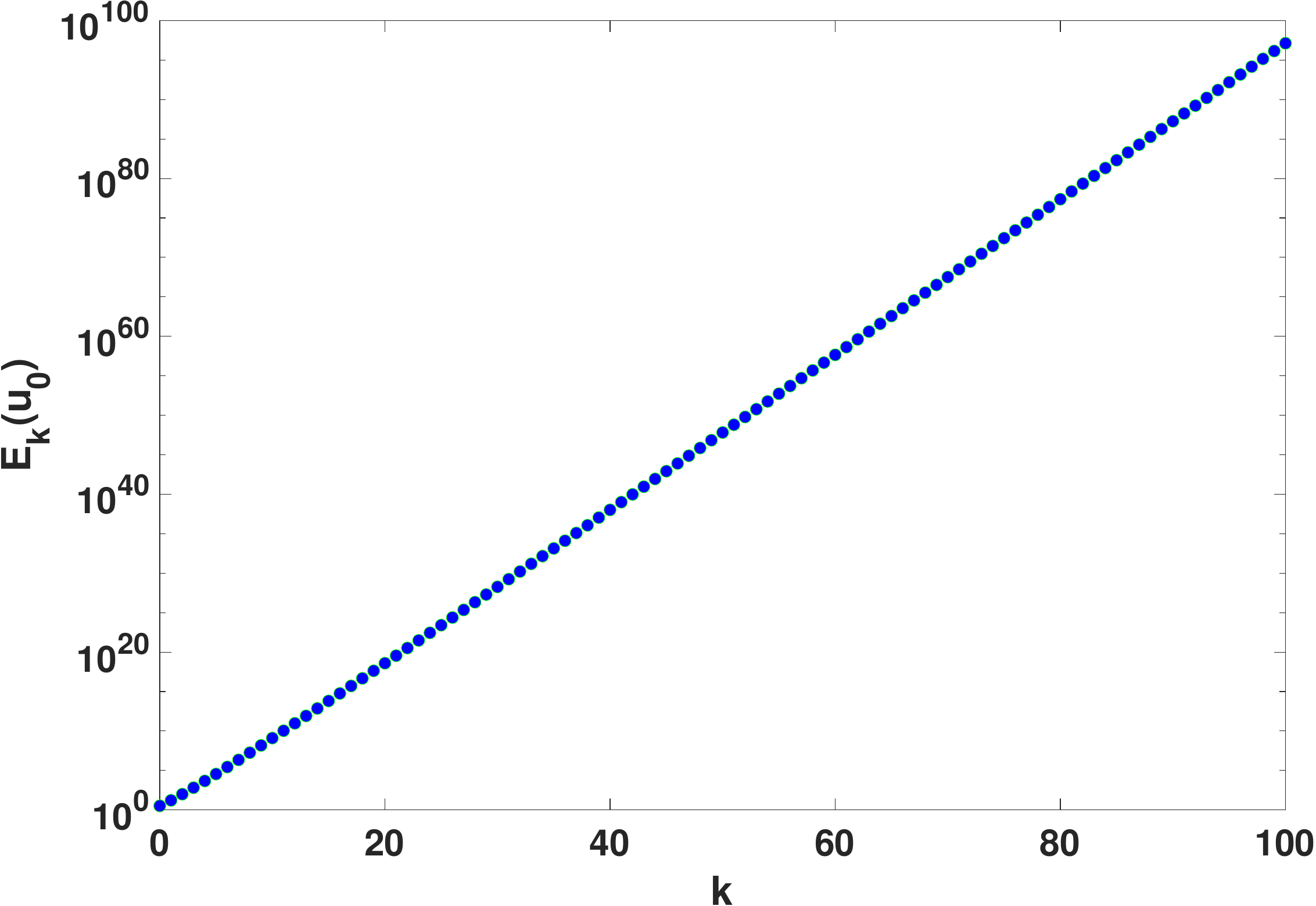}\\
		\textbf{(c)} \hspace{5cm} \textbf{(d)}\\
		\includegraphics[scale=0.14]{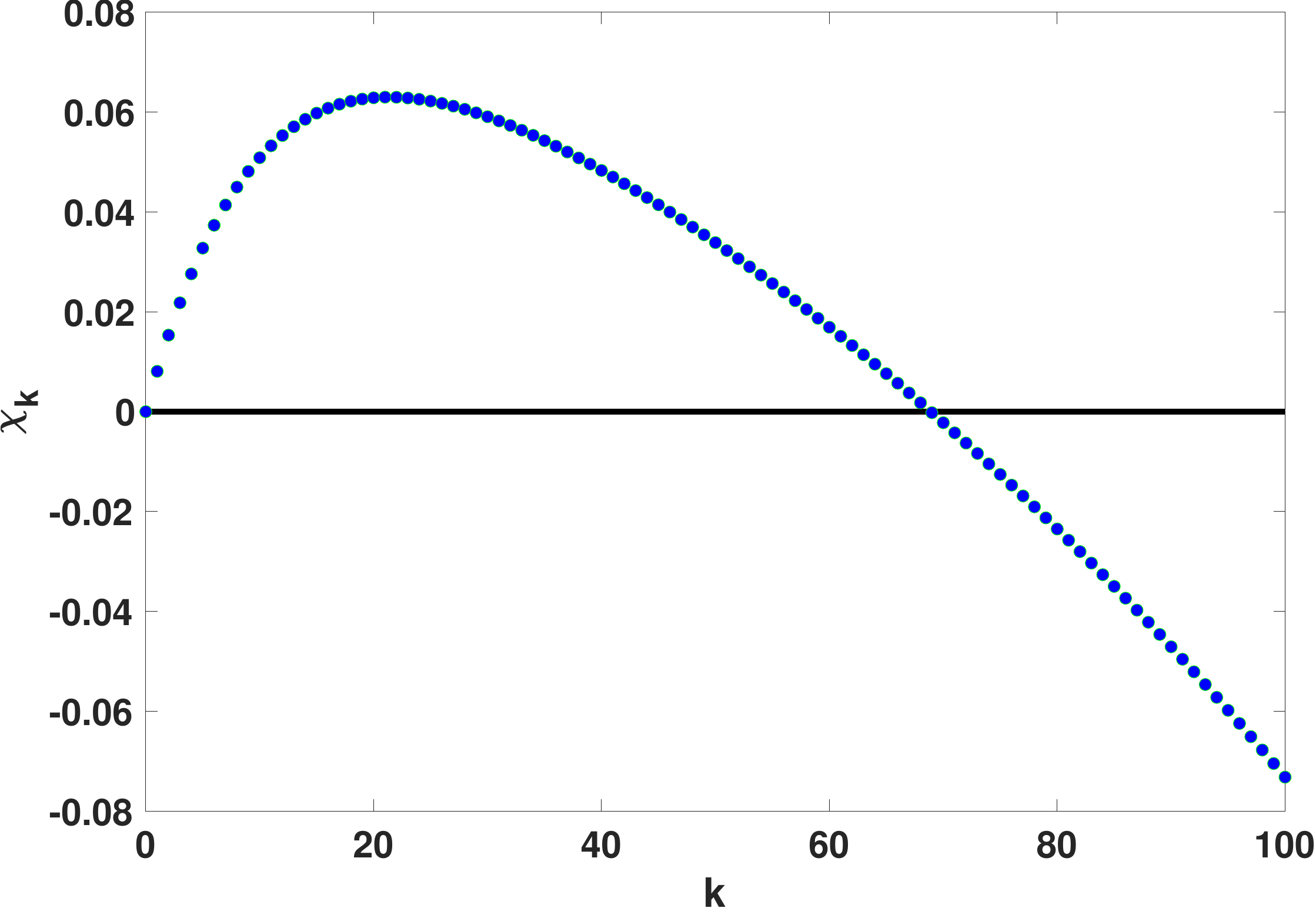}
		\includegraphics[scale=0.14]{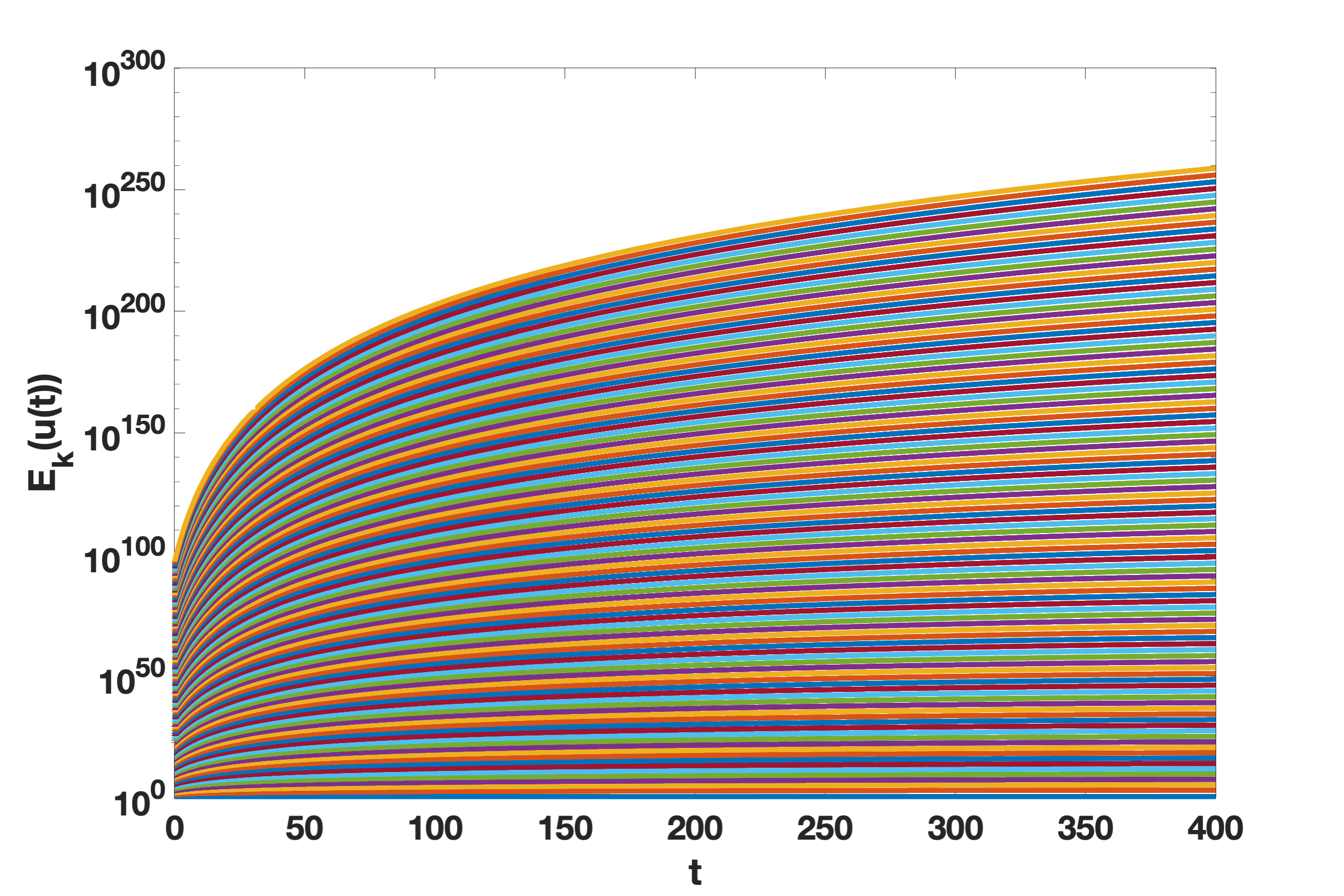}
	\end{center}
	\caption{\textit{We plot  (a) $b \mapsto u_0(b) $; (b) $k \mapsto E_k(u_0) $ for $k=0, \ldots, 100$ (with  y axis log scale); (c) $k \mapsto \chi_k $ for $k=0, \ldots, 100$; and (d) $t \mapsto E_k(u(t)) $ for $k=0, \ldots, 100$ (with y axis log scale);   }}\label{Fig3}
\end{figure}

%

\section{Discussion} 
\label{Section7}

This article proposes a new class of models describing the aging process. The model is based on the notion of biological age, which is a quantity reflecting the aging due to cells failing to repair DNA damages, illness, injuries, or, more generally speaking, corresponding to the body's aging.  The key features of the model are the following: 

\begin{itemize}
	\item[\rm 1)]  a drift term with a constant velocity  describing the aging process at the cellular level;
	
	\item[\rm 2)]  a rejuvenation operator describing the repair, recovery and healing processes during life; 
	
	\item[\rm 3)]   a premature aging operator corresponds to the medical care of injuries and illness.  
\end{itemize}
In this work, we consider the simplest version of the model. The model can be extended in several directions. To conclude the paper, we propose some possible extensions. 

\subsection{Aging model with birth and death processes}
The full model with both rejuvenation and premature aging with birth and death processes is the following
\begin{equation}  \label{6.1}
	\left\lbrace
	\begin{array}{rl}
		\partial_t u(t,b)+\partial_b u(t,b) &= -\mu(b) \, u(t,b)+ \beta(t) \, \Gamma(b),\vspace{0.2cm}\\
		& \quad  - (\tau_-+\tau_+) u(t,b) \vspace{0.2cm}\\
		&\quad + \tau_- \, g_-  \,  u\left(t, g_- \, b\right) \vspace{0.2cm}\\
		&\quad +\tau_+ \,  g_+ \,  u\left(t, g_+ \, b\right) \vspace{0.2cm}\\
		u(t,0)&=0,\\
		u(0,b)&=u_0(b) \in L^1_+\left((0,\infty),\R \right).
	\end{array}
	\right.
\end{equation}
In the above model, the function $-\mu(b)$ is the mortality rate for individuals with biological age $b$.  That is 
$$
\exp \left(- \int_{b_1}^{b_2} \mu \left(b\right) \d b \right)
$$
is the probability for an individual to survive from the biological $b_1$ to the biological age $b_2$. 

The term $ \beta(t)$ is the flow of new born at time $t$, that is 
$$
\int_{t_1}^{t_2}  \beta \left(\sigma\right) \d \sigma 
$$ 
is the number of newborn between $t_1$ and $t_2$.

The map $b \to  \Gamma(b)$ is the density of probability to have a biological age $b$ at birth. That is 
$$
\int_{b_1}^{b_2} \Gamma(b) db
$$
is the probability to obtain a newborn with biological age $b$ between $b_1$ and $b_2$. Moreover, 
$$
\int_{0}^{\infty} \Gamma(b) db=1. 
$$ 

\begin{assumption} \label{ASS6.1} We assume that $b \mapsto \mu(b)$ and $t \mapsto \beta(t)$ are constant functions. We also assume that 
	$$
	\mu >0 \text{ and } \beta >0. 
	$$
	We assume that the biological age of newborns follows a gamma density of probably. That is, 
	$$
	\Gamma(b)= \dfrac{\delta^\alpha \, b^{\alpha-1} \exp \left(-\delta \, b\right)}{\left(\alpha-1 \right)!} \text{ for } b>0,
	$$
	where $\alpha>1$, $\delta>0$, and the gamma function $	\left(\alpha-1 \right)!$ corresponds here to a constant of normalization, and is defined by 
	$$
	\left(\alpha-1 \right)!=\int_{0}^{\infty}   b^{\alpha-1} \exp \left(- b\right) \d b.
	$$
\end{assumption}
In Appendix \ref{AppendixA}, we discuss a aging model with birth and death processes. 
\subsection{Model with generalized jumps functions}
The full model with both rejuvenation and premature aging reads as follows 
\begin{equation}  \label{8.1}
	\left\lbrace
	\begin{array}{rl}
		\partial_t u(t,b)+\partial_b u(t,b) &= \ - (\tau_-+\tau_+) u(t,b) \vspace{0.2cm}\\
		&\quad + \tau_- \, f_-'(b)  \,  u\left(t, f_-(b) \right) \vspace{0.2cm}\\
		&\quad +\tau_+ \,  f_+'(b)  \,  u\left(t, f_+(b)\right) \vspace{0.2cm}\\
		u(t,0)&=0, \vspace{0.2cm}\\
		u(0,b)&=u_0(b) \in L^1_+\left((0,\infty),\R \right).
	\end{array}
	\right.
\end{equation}
We can extend the above model by setting 
$$
f_-(b)=\left( 1-\delta_-(b) \right) b \leq b \leq 	\left( 1+\delta_+(b) \right) b =f_+(b). 
$$ 
To assure the total mass preservation of the model we assume that  
$$
f_-(0)=f_+(0)=0
$$
and to preserve the positivity of the solutions we assume that 
$$
f_-'(b) \geq 0 \text{ and } f_+'(b)  \geq 0, \forall b \geq 0. 
$$
For example, we could use 
$$
f_+(b)=\left( 1+\delta_+ b^m \right) \, b, \text{ with } m\geq 0,
$$ 
where $\delta_+(b)=\delta_+ b^m$ could be any positive polynomial in $b$, that would model the average amplitude of  rejuvenation jumps.

The premature aging jumps $f_-(b)$ must remain below $b$ (the biological age after the jump) therefore, we could use 
$$
f_-(b)=\left( 1- \dfrac{\delta_-}{1+\chi \, b^m} \right) \, b, \text{ with } m\geq 0 \text{ and }0\leq \chi \leq \delta_-<1. 
$$

\subsection{Model with both chronological age and biological age}
It is unrealistic to get younger than $20$ years because puberty and post-pubertal growth are irreversible events, nor after $80$ years when the stock of undifferentiated stem cells is exhausted. Likewise, it is unreasonable to exceed the age of more than $120$ years, which is the physiological limit of the human lifespan. Hence the rate of rejuvenation or premature aging should not be constant in the function of the chronological age.  Therefore it would be important for this problem to consider both the biological and chronological ages. 

\medskip 
We  could also combine the biological and chronological age. Consider $a$ the chronological age (i.e. the time since birth) then we can combine both the chronological and the biological age and we obtain 
\begin{equation}   \label{8.2}
	\left\lbrace
	\begin{array}{rl}
		\partial_t u+\partial_a u+\partial_b u& =-\mu(a,b) u(t,a,b)  \vspace{0.2cm}\\
		&- \tau_+ u(t,a,b)+ \tau_+ \left( 1+\delta_+ \right)  u\left(t,a, \left( 1+\delta_+ \right) b\right) \vspace{0.2cm}\\ 
		&- \tau_- u(t,a,b)+ \tau_- \left( 1-\delta_- \right)  u\left(t,a, \left( 1-\delta_- \right) b\right)  \vspace{0.2cm}\\
		u(t,0,b)&=b(t) \Gamma(b), \text{ for } b \geq 0, \vspace{0.2cm}\\
		u(t,a,0)&=0, \text{ for } a \geq 0,
	\end{array}
	\right.
\end{equation}
with an initial distribution 
$$
u(0,a,b)=u_0(a,b) \in L^1_+\left((0,\infty)^2,\R \right).
$$
The function $u(t,a,b)$ is the density of population at time $t$ with respect to the chronological age $a$ and the biological age $b$. This means that if $0 \leq a_1 \leq a_2$ and $0 \leq b_1 \leq b_2$ then 
$$
\int_{a_1}^{a_2}\int_{b_1}^{b_2}u(t,a,b) \, da\,db
$$
is the number of individuals with chronological age $a$ in between $a_1$ and $a_2$ and the biological $b$ in between $b_1$ and $b_2$. 

This problem can be reformulated as follows 
\begin{equation}   \label{8.2}
	\left\lbrace
	\begin{array}{rl}
		\partial_t u+\partial_a u& =Au(t,a,b)-\mu(a,b) u(t,a,b)  \vspace{0.2cm}\\
		&- \tau_+ u(t,a,b)+ \tau_+ \left( 1+\delta_+ \right)  u\left(t,a, \left( 1+\delta_+ \right) b\right) \vspace{0.2cm}\\ 
		&- \tau_- u(t,a,b)+ \tau_- \left( 1-\delta_- \right)  u\left(t,a, \left( 1-\delta_- \right) b\right)  \vspace{0.2cm}\\
		u(t,0,b)&=b(t) \Gamma(b), \text{ for } b \geq 0, 
	\end{array}
	\right.
\end{equation}
We refer to Magal and Ruan \cite{Magal-Ruan} for more results about age-structured models combined with an extra structuring variable.

\vspace{1cm}
\appendix
\begin{center}
	{\LARGE	\textbf{Appendix}}
\end{center}
\section{Aging model with birth and death processes}
\label{AppendixA}

\medskip 
\noindent \textbf{Abstract Cauchy problem reformulation:} The problem \eqref{6.1} can be reformulated as an abstract Cauchy problem 
\begin{equation} \label{6.2}
	\left\{ 
	\begin{array}{l}
		u'(t)=Au(t)+ B u(t)- \mu u(t) +\beta \, \Gamma(.), \text{ for } t \geq 0,\vspace{0.2cm} \\
		\text{with}\vspace{0.2cm} \\
		u(0)=u_0 \in L^1_+ \left(\left(0, \infty\right), \R \right).	
	\end{array}
	\right.
\end{equation}
where 
$$
Bu= \tau_- B_{g_-}u+ \tau_+ B_{g_+}u-\left(\tau_-+\tau_+\right) u.
$$
\begin{theorem} \label{TH6.2}Let Assumption \ref{ASS5.1} and Assumption \ref{ASS6.1} be satisfied. Then  the mild solution \eqref{6.2} is given by 
	\begin{equation*}
		u(t)=T_{A+B-\mu I}(t)u_0+ \int_{0}^{t}  T_{A+B-\mu I}(t-\sigma) \beta \, \Gamma(.) \d \sigma, \forall t \geq 0, 
	\end{equation*} 
	where 
	$$
	T_{A+B-\mu I}(t)=e^{-\mu t}T_{A+B}(t), \forall t \geq 0.
	$$
	Moreover 
	$$
	\lim_{t \to \infty} u(t)= \overline{u} \geq 0,
	$$
	where 
	\begin{equation} \label{6.3}
		\overline{u}= \int_{0}^{\infty}  T_{A+B-\mu I}(\sigma) \beta \, \Gamma(.) \d \sigma.	
	\end{equation}
\end{theorem}

\noindent \textbf{Moments formulation:} We obtain the following result using similar arguments to Theorem  \ref{TH3.6}. 

By using change of variable, we obtain  
$$
E_k(\Gamma)=\int_{0}^{\infty} \dfrac{\delta^\alpha b^{\alpha+k-1} \exp \left(-\delta b\right)}{\left(\alpha-1 \right)!} \d b =\delta^{-k}\dfrac{\left(\alpha+k-1 \right)!}{\left(\alpha-1 \right)!}. 
$$
\begin{theorem} \label{TH6.3} Let Assumption \ref{ASS5.1} be satisfied.  The rejuvenation and premature aging model \eqref{5.1} has a unique non-negative mild solution.  We have  
	\begin{equation*} 
		\frac{d}{dt}	E_0(u(t)) =-\mu \,	E_0(u(t)) + \beta, 
	\end{equation*}
	and the model preserves the total mass (number) of individuals. That is 
	\begin{equation} \label{6.4}
		\lim_{t \to +\infty} E_0(u(t)) =   \dfrac{ \beta}{ \mu}. 
	\end{equation}
	Moreover, the higher moment satisfies the following system of ordinary differential equations for each $k \geq 1$, 
	\begin{equation}  \label{6.5}
		\frac{d}{dt}	E_k(u(t)) =k \, E_{k-1}(u(t))- \left( \mu+\chi_k \right) \, E_{k}(u(t))+\beta \, \delta^{-k}\dfrac{\left(\alpha+k-1 \right)!}{\left(\alpha-1 \right)!}.
	\end{equation}
	
\end{theorem} 

\begin{proposition}\label{PROP6.4} Let $k_1$ be an integer such that 
	\begin{equation*}
		\chi_k<-\mu, \forall k > k_1.	
	\end{equation*}
	Then 
	\begin{equation}  \label{6.6}
		\lim_{t \to +\infty} E_k(u(t)) =+\infty,  \forall k >  k_1,
	\end{equation}	
	whenever $ E_k(u_0) <+\infty$. 
	
\end{proposition} 

\medskip 
\noindent \textbf{Equilibrium solution:}  An equilibrium solution satisfies some delay equation with both advance and retarded delay. That is,  
\begin{equation}  \label{6.7}
	\left\lbrace
	\begin{array}{rl}
		\overline{u}'(b) &= \beta \, \Gamma(b) - \left( \mu+\tau_-+\tau_+ \right) \overline{u}(b) \vspace{0.2cm}\\
		&\quad + \tau_- \, g_-  \,  \overline{u}\left(g_- \, b\right) \vspace{0.2cm}\\
		&\quad +\tau_+ \,  g_+ \,   \overline{u} \left(g_+ \, b\right), \vspace{0.2cm}\\
		\overline{u}(0)&=0,  
	\end{array}
	\right.
\end{equation}
and the difficulty of solving such an equation comes from the following 
$$
g_+ \, b>b >g_- \, b, \, \forall b >0.
$$ 
It follows that, even the existence of equilibrium solution is not a classical problem to investigate. 

\medskip 
\noindent \textbf{Non existence result for exponentially decreasing equilibrium solution of \eqref{5.1}:}  Assume that $b \mapsto \overline{u}(b) $ a non-negative continuously differentiable map satisfying system \eqref{5.10}.  

\begin{assumption} \label{ASS6.5}
	Assume that the map $b \mapsto \overline{u}(b) $ is non-negative, and non null, and  continuously differentiable map, and  satisfies the system \eqref{6.7}. Assume in addition that $ \overline{u}(b) $ is exponentially decreasing. That is, there exist two constants  $M>0$, and $\gamma>0$, such that 
	\begin{equation} \label{6.8}
		\overline{u}(b) \leq M e^{-\gamma b}, \forall b \geq 0. 	
	\end{equation}
\end{assumption}
By using  the first equation of \eqref{5.10}, we deduce that 
$$
\vert \overline{u}'(b) \vert \leq \widetilde{M} e^{-\gamma \, \min(1,g_+, g_-, \delta/2) \, b}, \forall b \geq 0,
$$
for some suitable $\widetilde{M} >0$.

So, under Assumption \ref{ASS5.7},  all the moments of  $\overline{u}(b) $ and $\overline{u}'(b) $ are well defined. Moreover, by using the first equation of  \eqref{5.10}, we obtain for each $k\geq 1 $ 
\begin{equation*}
	\begin{array}{rl}
		\displaystyle		\int_{0}^{\infty} \sigma^k \overline{u}'(\sigma) \d \sigma &= 	\int_{0}^{\infty} \beta \, \Gamma(b) \d b  - \left(\mu +\tau_-+\tau_+ \right) 	\int_{0}^{\infty}  \sigma^k \overline{u}(\sigma) \d \sigma \\
		
		&\quad  + \tau_+  	\int_{0}^{\infty} \sigma^k \overline{u}\left(t, g_+  \sigma \right)   g_+   \d \sigma\\
		&\quad + \tau_-  	\int_{0}^{\infty} \sigma^k \overline{u}\left(t, g_-  \sigma \right)   g_-   \d \sigma,
	\end{array}		
\end{equation*}
and since $\overline{u}(0)=0 $, we obtain by integrating by parts 
\begin{equation*}
	\begin{array}{rl}
		\displaystyle	-k	\int_{0}^{\infty} \sigma^{k-1} \overline{u}(\sigma) \d \sigma &=  \beta \, \delta^{-k}\dfrac{\left(\alpha+k-1 \right)!}{\left(\alpha-1 \right)!} - \left(\mu +\tau_-+\tau_+ \right) 	\int_{0}^{\infty}  \sigma^k \overline{u}(\sigma) \d \sigma \\
		
		&\quad  + \tau_+  	\int_{0}^{\infty} \sigma^k \overline{u}\left(t, g_+  \sigma \right)   g_+   \d \sigma\\
		&\quad + \tau_-  	\int_{0}^{\infty} \sigma^k \overline{u}\left(t, g_-  \sigma \right)   g_-   \d \sigma,
	\end{array}		
\end{equation*}
and we obtain 
\begin{equation} \label{6.9}
	0=\beta \, \delta^{-k}\dfrac{\left(\alpha+k-1 \right)!}{\left(\alpha-1 \right)!} +	k	E_{k-1}(\overline{u})- \left( \chi_k+ \mu \right)  \, E_{k}(\overline{u}), \forall k \geq 1, \forall k \geq 1, 
\end{equation}
where $\chi_k$ is defined by \eqref{5.5}. 

Under Assumption \ref{ASS6.5}, we must have 
$$
E_{k-1}(\overline{u})>0, \text{ and } E_{k}(\overline{u})>0 , \forall k \geq 0,
$$
and by using \eqref{5.6}, we deduce that \eqref{6.9} can not be satisfied for all $k \geq 1$ large enough. Because  $ -\left( \chi_k+ \mu \right)  >0,$ for all $k \geq 1$ large enough. 

Therefore we obtain the following proposition. 

\begin{proposition} \label{PROP6.6}
	The equilibrium $\overline{u}$ defined by \eqref{6.3}  is  not an exponentially decreasing function. That is, for each  $M>0$, and $\gamma>0$,, the function $\overline{u}$ does not satisfy \eqref{6.8}.  
\end{proposition}
The above proposition is surprising because the proposition shows that it is sufficient to perturb an age-structured with a non-local term $ - \tau_- u(t,b)+ \tau_- \left( 1-\delta_- \right)  u\left(t, \left( 1-\delta_- \right) b\right)$ to obtain an equilibrium solution which is not exponential bounded.

\bibliographystyle{aims}

\bibliography{bibliography}

\end{document}